\documentclass[onefignum,onetabnum]{siamart190516}

\usepackage{lipsum}
\usepackage{amsmath,amssymb,amsfonts,latexsym,stmaryrd}
\usepackage{graphicx,float}
\usepackage{mathrsfs} 
\usepackage{color}
\usepackage{epstopdf}
\usepackage{algorithmic}
\usepackage{multirow}
\ifpdf
\DeclareGraphicsExtensions{.eps,.pdf,.png,.jpg}
\else
\DeclareGraphicsExtensions{.eps}
\fi

\def\O{\Omega}

\newcommand{\abs}[1]{\lvert#1\rvert}

\usepackage{booktabs}
\usepackage{float}

\newcommand\bu{\boldsymbol{u}}
\newcommand\bv{\boldsymbol{v}}
\newcommand\bw{\boldsymbol{w}}

\newcommand\bn{\boldsymbol{n}}

\def\CT{{\mathcal T}}

\renewcommand\H{\mathrm{H}}

\renewcommand\O{\Omega}

\renewcommand\div{\mathop{\mathrm{div}}\nolimits}

\newcommand{\vertiii}[1]{{\left\vert\kern-0.25ex\left\vert\kern-0.25ex\left\vert #1 
		\right\vert\kern-0.25ex\right\vert\kern-0.25ex\right\vert}}

\newsiamremark{remark}{Remark}
\newsiamremark{hypothesis}{Hypothesis}
\crefname{hypothesis}{Hypothesis}{Hypotheses}
\newsiamthm{claim}{Claim}

\headers{A finite element model for membrane channel }{Nicolás Carro, David Mora and Jesus Vellojin}

\title{A finite element model for concentration polarization and osmotic effects in a membrane channel \thanks{Submitted to the editors DATE.
		\funding{The authors were partially supported by ANID-Chile through project Anillo of Computational Mathematics for Desalination Processes ACT210087. The second author was partially supported by DICREA through project 2120173 GI/C Universidad del Bío-Bío, by the National Agency for Research and Development, ANID-Chile through FONDECYT project 1220881, and by project Centro
			de Modelamiento Matemático (CMM), ACE210010 and FB210005, BASAL funds for centers of excellence. 
}}}

\author{Nicolás Carro\thanks{GIMNAP-Departamento de Matem\'atica, Universidad del B\'io - B\'io, Casilla 5-C, Concepci\'on, Chile. \email{ncarro@ubiobio.cl}.} \and David Mora\thanks{GIMNAP-Departamento de Matem\'atica, Universidad del B\'io - B\'io, Casilla 5-C, Concepci\'on, Chile. \email{dmora@ubiobio.cl}.}
	\and Jesus Vellojin\thanks{Corresponding author. GIMNAP-Departamento de Matem\'atica, Universidad del B\'io - B\'io, Casilla 5-C, Concepci\'on, Chile. \email{jvellojin@ubiobio.cl}.}}

\usepackage{amsopn}

\ifpdf
\hypersetup{
	pdftitle={A Finite element model for cross-flow membrane channel},
	pdfauthor={Nicolás Carro, David Mora, and Jesus Vellojin}
}
\fi

\begin{document}
	
	\maketitle
	
	\begin{abstract}
		In this paper we study a mathematical model that represents the concentration polarization and osmosis effects in a reverse osmosis cross-flow channel with porous membranes at some of its boundaries. The fluid is modeled using the Navier-Stokes equations and Darcy's law is used to impose the momentum balance on the membrane. The scheme consist of a conforming finite element method with the velocity–pressure formulation for the Navier-Stokes equations, together with a primal scheme for the convection–diffusion equations. The Nitsche method is used to impose the permeability condition across the membrane. Several numerical experiments are performed to show the robustness of the method. The resulting model accurately replicates the analytical models and predicts similar results to previous works. It is found that the submerged configuration has the highest permeate production, but also has the greatest pressure loss of all three configurations studied.
	\end{abstract}
	
	
	
	\section{Introduction}
	Membrane-based desalination methods have been widely used for tackling the water scarcity world problem, due to its relatively low energy consumption compared to thermal-based methods such as multi-stage flash \cite{skuse2021can}, and its capability to use renewable or low-grade energy resources \cite{ahmed2021emerging,xu2020heat}. Currently, reverse osmosis leads this trend, being used in $69\%$ of industrial desalination plants around the world \cite{eke2020global}, as it is the most energy-efficient desalination process and its membrane technology is the most investigated, accounting for more than $60\%$ of articles concerning the most relevant membrane-based desalination techniques in 2019 \cite{lee2020fouling}. 
		
	Models associated with concentration polarization in reverse osmosis (RO) are usually based on the Navier-Stokes and convection-diffusion equations, although the Brinkman equations for porous media can also be considered. Variations in the approach of these equations for the formulation of concentration polarization equations have been widely discussed due to the difficulty in obtaining a numerical solution.  Geometric parameters (such as the length and thickness of the channel or the existence of spacers) and physical parameters (low density and high diffusivity) are associated with unstable behavior of the numerical methods proposed. This is usually improved by considering sufficiently fine meshes, at the cost of increasing the memory space required to run the algorithm, or using the SUPG method, which is capable of stabilize the scheme for convective dominated problems \cite{brooks1982streamline}. Although naturally different, the equations modeling RO have a clear similarity with the Boussinesq equations, and that is the coupling between fluid and convection-diffusion models. Behind the Boussinesq equations there is an extensive bibliography, analysis and modeling, where different configurations and environmental conditions are assumed in order to analyze their influence in the development of numerical methods for their solution (see \cite{allali2005priori,bernardi1995couplage,colmenares2016analysis,colmenares2017augmented} and the references therein). However, the consideration of a nonlinear boundary condition to model the water flow in the membrane makes the difference between both models. 
	
		A major difficulty in these systems is the boundary conditions, which can significantly complicate the study of a problem, especially if they are boundary conditions on vector functions. In the models associated with RO we can observe that there are boundary conditions, such as the permeability associated with the membrane in terms of the osmotic pressure, which present a nonlinear behavior. The numerical implementation always represents a challenge at the computational level because the numerical scheme selected to solve the problem must be adjusted to consider this condition. Now, several techniques have been proposed over the years to deal with boundary conditions on vector fieds, such as the slip condition in the Navier-Stokes equations, which has been implemented using penalty methods, Lagrange multipliers, or the Nitsche technique \cite{dione2015penalty,freund1995weakly, urquiza2014weak,verfurth1986finite,verfurth1991finite}.  

	The Nitsche technique was introduced in 1971 \cite{nitsche1971variationsprinzip} as an idea to impose Dirichlet conditions without the use of Lagrange multipliers. The scheme is similar to a mesh-size-dependent penalty method, but consistency terms are added for optimal convergence. Over the years, its use has been extended to interface problems, contact, and boundary conditions in general.
	
	At the discrete level, the majority of RO models is simulated with Finite Difference (FD) or Finite Volumes methods (FV), thus little has been researched in finite elements (FE) numerical algorithms applied to this kind of problems outside the preexisting commercial software, e.g. COMSOL \cite{picioreanu2009three,li2016three,gu2017effect,el2020numerical,perfilov2018general,parasyris2020mathematical} or ANSYS CFX \cite{liang2014cfd,weihs2014cfd,keshtkar2020novel}, where boundary implementation is done through a user interface or user defined macros. The coupling between velocity and concentration in RO models are treated in a iterative way by many CFD softwares using FE or FV. For example, we have the works of Wiley and Fletcher \cite{wiley2003techniques,alexiadis2007cfd,fletcher2004computational}, where they implemented a CFD model for fluid and transport processes. Here, it is observed that CFD software are  versatile when complex geometries are considered.  There is also the work of Salcedo-Diaz et al. [44], where the use of Comsol Multiphysics were use to  study the behavior in a three dimensional model of the coupled system, and a ﬂuid -dynamics and solute transfer numerical simulation that was studied through the Ansys software \cite{lyster2009coupled}. 
	
	Some of the ideas behind this CFD softwares uses the approach made in \cite{ma20042}, where we observe one of the first stabilized finite element approximations on a RO channel. Here, the membrane dimensions are taken into account, together with a SUPG stabilization for coercivity conservation in the convective dominated model of the coupling between Navier-Stokes and convection-diffusion equations. An alternative approach was studied in \cite{bernales2017prandtl} a numerical method is developed based on the Prandtl equations that couples laminar fluid and mass transport for the study of concentration polarization and osmotic effects in a membrane. The validation of the model is done by comparison with classical analytical solutions.
	
	In our study, the goal is to propose a finite element method using the Nitsche technique on an RO model. The difference with respect to other studies in the literature is that the Nitsche trick allow to impose the Darcy-Starling law by means of a consistent finite element method. This proposition is inspired by \cite{juntunen2009nitsche},  along with several studies on Nitsche's method on interface and contact problems \cite{chouly2013convergence,chouly2015nitsche1,chouly2015nitsche2}, where it is shown that Nitsche's technique can be applied to problems with general boundary conditions. The main contribution of this work is that the resulting discrete model is easy to implement because linearizing the model depends only on linearizing the flow. Furthermore, to the best of the authors' knowledge, this is the first time that the Nitsche technique is used to impose a nonlinear boundary condition on a cross-flow model that depends on the unknowns. The model is implemented with the open source software FEniCS \cite{AlnaesEtal2015}, and the results are validated with analytical solutions of simplified problems.
	
The structure of this article is as follows: in section 2 the assumptions for the studied system are stated, and the transport equations with the respective boundary conditions. The weak formulation for the continuous model is given in section 2.1. The discretization of the model using the Nitsche method is given in section 2.3. Then, the numerical results for various operating conditions and the discussion of the obtained results are presented in section 3.


	\section{Model problem}
	Consider the domain $\Omega=[0,L]\times[0,d]$, which consists of a channel with semi-permeable walls at $\Gamma_m$ and circular spacers at $\Gamma_w$, with flow input in $\Gamma_{\text{in}}$ and output in $\Gamma_{\text{out}}$ (see Figure \ref{fig:dibujo_canal}). The assumptions for this model are: 
	\begin{enumerate}
		\item The system is in steady state.
		\item The system is isothermal (therefore, the heat equation in the system is omitted).
		\item The fluid in the channel representing seawater is only composed of water and salt (sodium chloride). Therefore, the fouling effects on the membrane's performance due to scaling of less soluble salts or biofouling caused by bacteria proliferation, are ignored.
		\item The fluid is considered to be Newtonian and incompressible with density $\rho_0$. 
		\item The fluid has uniform viscosity $\mu_0$ and solute diffusivity through the solvent $D_{0}$.
		\item The channel membranes are porous in nature, and its behavior is described by Darcy-Starling Law.
		\item The osmotic pressure effect of the dissolved solute is represented by Van't Hoff's Law. 
		\item The effect of pressure drop inside the channel due to viscous effects on the permeate flux (Darcy-Starling Law) is negligible. Therefore, $\Delta P$ in equation \ref{model-eq9} can be considered constant. As an approximation, we will consider this valid for a channel whose length $L$ is such that the pressure loss due to viscous effects is less than $1\%$ the value of $\Delta P$.
		\item The membrane's performance does not become affected by wearing, so $I_{0}$ is constant.
	\end{enumerate}
	\begin{figure}[h!]
		\centering
	\includegraphics[scale=1.7]{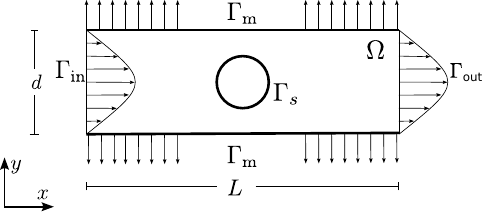}
	\caption{Model of cross-flow membrane filtration and concentration polarization.}
	\label{fig:dibujo_canal}
	\end{figure}
	
	For the described model, we will use a Navier-Stokes (based on assumptions 4 and 5) and a convection-diffussion model (based on assumptions 3 and 5), whose coupling is well-known as the Boussinesq equations. To make matters precise, we have
	
	\textit{The continuity condition:}
	\begin{equation}
		\div\,\bu=0, \text{ in }\Omega. \label{model-eq1}
	\end{equation}

	\textit{The Navier-Stokes equation}
	\begin{equation}
		\rho_0\bu\cdot\nabla\bu=-\nabla p +\mu_0\Delta\bu, \text{ in }\Omega.\label{model-eq2}\\
	\end{equation}
		
\textit{The Convection-Diffusion equation:}
\begin{equation}
	\bu\cdot\nabla \theta=D_0\Delta \theta, \text{ in }\Omega. \label{model-eq3}\\
\end{equation}
With respect to the boundary conditions,  we assume an inlet parabolic velocity profile with a fixed concentration to simulate a fully developed flow of seawater entering the channel, a normal momentum balance to set the pressure and a zero diffusive salt flux boundary condition in the outlet. Additionally, a zero fluid velocity, together with no salt penetration are assumed in the spacers or impermeable walls. This is summarized in the following equations
	\begin{align}	
	\bu\cdot\bn&=6u_{0}\frac{y}{d}(1-\frac{y}{d}),\;\theta=\theta_0,&\text{ on }\Gamma_{\text{in}},\label{model-eq4}\\
	\mu_0\frac{\partial\bu}{\partial\bn}-p\bn&=\boldsymbol{0},&\text{ on }\Gamma_{\text{out}},\label{model-eq5}\\
	-D_0(\nabla \theta)\cdot\bn&=0,&\text{ on }\Gamma_{\text{out}},\label{model-eq6}\\
	\bu&=\boldsymbol{0},&\text{ on }\Gamma_s,\\
	\theta(\bu\cdot \bn)-D_0\nabla \theta\cdot \bn&=0,&\text{ on }\Gamma_m,\Gamma_s, \label{model-eq7}\\
	\bu\cdot\boldsymbol{t}&=0,&\text{ on }\Gamma_m,\label{model-eq8}\\
	\bu\cdot\bn&=\frac{\Delta P - \kappa \theta}{I_0},&\text{ on }\Gamma_m.\label{model-eq9}
\end{align}
In the above, $\bu=(u_x,u_y)$, $p$, and $\theta$ represent the fluid velocity, pressure and molar concentration profile, respectively. Its worth noting that the pressure $p$ is the gauge pressure and not the absolute pressure $P$, as the assumption $4$ allows to write the Navier-Stokes equations in terms of $p$. The constant $D_0$ represents the solute diffusivity trough the solvent, $I_0$ is the membrane resistance, whereas $\kappa$ is a constant given by the Darcy-Starling law. $\Delta P$ represents the transmembrane pressure between the membrane and the ``outside" (the permeate channel).
	
	\subsection{Formulation}
	Assuming free flow conditions in $x=L$, we define the following spaces.
	$$
	\begin{aligned}
		&\H:=\{\bv\in H^1(\Omega)^2\;:\; \bv\cdot\bn=u_0 \text{ on }\Gamma_{\text{in}}\},\\
		&Y:=\{\bv\in H^1(\Omega)^2\;:\; \bv=\boldsymbol{0} \text{ on }\Gamma_{\text{in}}\cup\Gamma_s,\; \bv\cdot\boldsymbol{t}=\boldsymbol{0}\text{ on }\Gamma_m\}, \qquad Q=L^2(\Omega),\\
		&Z:=\{\tau\in H^1(\Omega)\;:\;  \tau=\theta_0 \text{ on }\Gamma_{\text{in}} \},\qquad		X:=\{\tau\in H^1(\Omega)\;:\;  \tau=0 \text{ on }\Gamma_{\text{in}} \}.
	\end{aligned}
	$$
	
	Testing \eqref{model-eq2} with $\bv\in Y$, integrating by parts and using the boundary conditions we obtain
	$$
	\begin{aligned}
		\int_\Omega \rho_0\left(\bu\cdot\nabla \bu\right)\cdot\bv&= \left[\int_\Omega p \div\bv-\int_{\partial\Omega} (p\bn)\cdot\bv\right] +\mu_0\left[-\int_\Omega\nabla\bu:\nabla \bv + \int_{\partial\Omega}\frac{\partial\bu}{\partial\bn}\cdot\bv\right]\\
		&= -\int_\Omega\mu_0\nabla\bu:\nabla \bv+\int_\Omega p \div\bv+ \int_{\Gamma_m}\bn^t\left(\mu_0\frac{\partial\bu}{\partial\bn}-p\bn\right)(\bv\cdot\bn)
	\end{aligned}
	$$
	
	Testing \eqref{model-eq3} with $\tau\in X$ and integrating by parts we obtain
	$$
	\begin{aligned}
	\int_\Omega (\bu\cdot\nabla \theta)\tau&=D_0\left[-\int_\Omega\nabla \theta\cdot\nabla \tau + \int_{\partial\Omega}\frac{\partial \theta}{\partial n}\tau\right]\\
	&=-\int_\Omega D_0\nabla \theta\cdot\nabla \tau + \int_{\Gamma_m\cup\Gamma_s}\theta(\bu\cdot\bn)\tau.
	\end{aligned}
	$$
	
Lets introduce the following forms: 
\begin{equation}
	\label{eq:formas_bilineales_continuas}
	\begin{aligned}
		&a(\bu,\bv):=\mu_0\int_\Omega \nabla\bu:\nabla\bv,\qquad\widetilde{a}(\bw,(\bu,\bv)):=\rho_0\int_\Omega (\bw\cdot\nabla\bu)\cdot\bv,\\
		&d(\theta,\tau):=D_0\int_\Omega \nabla \theta\cdot\nabla\tau, \qquad b(\bv,q):=-\int_\Omega q\div\bv\\
		 &\widetilde{c}(\bw,(\theta,\tau)):=\int_\Omega (\bw\cdot\nabla \theta)\tau-\int_{\Gamma_m\cup\Gamma_s}\theta\left(\bw\cdot \bn\right)\tau.
	\end{aligned}
\end{equation}

Hence, we have the following formulation for the proposed cross-flow model: Find $(\bu,p,\theta)\in \H\times Q\times Z$ such that \eqref{model-eq9} holds, and
\begin{equation}
	\label{eq:cross-flow-fv1}
	\begin{aligned}
		a(\bu,\bv)+b(\bv,p) + \widetilde{a}(\bu,(\bu,\bv)) - \int_{\Gamma_m}\bn^t\left(\mu_0\frac{\partial\bu}{\partial\bn}-p\bn\right)(\bv\cdot\bn)&=0,&\forall \bv\in Y,\\
		b(\bu,q)&=0,&\forall q\in Q,\\
		d(\theta,\tau) + \widetilde{c}(\bu,(\theta,\tau))&=0,&\forall \tau\in X.
	\end{aligned}
\end{equation}
Note that if homogeneous boundary conditions for both $\bu$ and $\theta$ are considered in $\Gamma_{m}$, then we fall into a typical formulation of Boussinesq equations. For an analysis on viscous flows around porous with non-homogeneous boundary data, we refer to \cite{ingram2011finite}. On the other hand, note that the nonlinearity inherited by the boundary condition \eqref{model-eq7} is present in the trilinear form $\widetilde{c}$. 
For the bilinear form $b(\cdot,\cdot)$ there holds \cite[Theorem 3.7]{girault1979finite}
\begin{equation}
	\label{eq:inf-sup-continua-b}
	\sup_{\bv\in Y,\bv\neq\boldsymbol{0}}\frac{b(\bv,q)}{\Vert\bv\Vert_{1,\Omega}}\geq\beta\Vert q\Vert_{0,\Omega},\qquad q\in W.
\end{equation}
\subsection{Nitsche method}
Let us consider a shape-regular family of partitions of $\O$, denoted by $\{\mathcal{T}_h\}_{h>0}$.  Let $h_T$ be the diameter of a triangle $T\in\CT_h$ and let us define $h:=\max\{h_T\,:\, T\in \CT_h\}$.

One of the main difficulties in the discrete solution of problems involving the Navier-Stokes equations is the choice of elements satisfying the discrete version of \eqref{eq:inf-sup-continua-b}. Of all the possible existing options, we choose the Taylor-Hood elements \cite{boffi1994stability,boffi1997three}, namely, given $k\geq 1$, the finite element spaces to develop the numerical scheme for the velocity and pressure are the following 
\begin{align*}
	\mathbf{H}_h&:=\{\bv_h\in C(\Omega)^2\,:\, \bv_h|_T\in\mathbb{P}_{k+1}(T)^2\quad\forall T\in\mathcal{T}_h\}\cap Y,\\
	Q_h&:=\{q_h\in C(\Omega)\,:\, q_h|_T\in\mathbb{P}_{k}(T)\quad\forall T\in\mathcal{T}_h\}.
\end{align*}
Another widely used and computationally very economical option is the lowest order mini-element pair $(\mathbf{H}_h,Q_h)$, where $Q_h$ is as above with $k=1$, and $\mathbf{H}_h$ is given by
$$
	\mathbf{H}_h:=\{\bv_h\in C(\Omega)^2\,:\, \bv_h|_T\in\mathbb{P}_{1,b}(T)^2\quad\forall T\in\mathcal{T}_h\}\cap Y,
$$ 
where $\mathbb{P}_{1,b}$ are the bubble-enriched piecewise continuous elements defined with the barycentric coordinates of each element $T$ (see, e.g.  \cite{ern2004applied,girault1979finite}).

In turn, for $\ell\geq 1$, the finite element space for the discrete concentration is just
$$
M_h:=\{\tau_h\in C(\Omega)\,:\, \tau_h|_T\in\mathbb{P}_{\ell}(T)\quad\forall T\in\mathcal{T}_h\}\cap X.
$$
It is important to mention that condition \eqref{model-eq9} is not imposed in throughout the integration by parts. Hence, we use the Nitsche's method, whose goal is to relax and
penalize Dirichlet boundary conditions such that they are satisfied in the limit of small mesh size. In this study, we will use this technique in order to impose the boundary condition \eqref{model-eq9}, which characterizes the pressure difference between the interior of the channel (feed) and the outside (permeate). In order to propose our numerical scheme using the Nitsche method, we introduce the  following bilinear forms 
\begin{multline}
	A((\bu_h,p_h),(\bv_h,q_h)):=
	a(\bu_h,\bv_h)+b(\bv_h,p_h)+ b(\bu_h,q_h) \\
	-\int_{\Gamma_m}\bn^t\left(\mu_0\frac{\partial\bu_h}{\partial\bn}-p_h\bn\right)(\bv_h\cdot\bn) - \int_{\Gamma_m}\bn^t\left(\mu_0\frac{\partial\bv_h}{\partial\bn}-q_h\bn\right) \left(\bu_h\cdot\bn \right)\\
	+\alpha\int_{\Gamma_m}h^{-1}(\bu_h\cdot\bn)(\bv_h\cdot\bn),\\
	F(\bv_h,q_h):=\alpha\int_{\Gamma_m} h^{-1}\left(\frac{\Delta P }{I_0}\right)(\bv_h\cdot\bn)-\int_{\Gamma_m}\bn^t\left(\mu_0\frac{\partial\bv_h}{\partial\bn}-q_h\bn\right)\left(\frac{\Delta P }{I_0}\right)\\
	B(\theta_h,(\bv_h,q_h)):=\alpha\int_{\Gamma_m} h^{-1}\left(\frac{\kappa \theta_h}{I_0}\right)(\bv_h\cdot\bn)-\int_{\Gamma_m}\bn^t\left(\mu_0\frac{\partial\bv_h}{\partial\bn}-q_h\bn\right)\left(\frac{\kappa \theta_h}{I_0}\right),
\end{multline}
where $h$ represents the local mesh size. The parameter $\alpha>0$ is a constant that has to be chosen sufficiently large. With this forms at hand, we propose the following Nitsche method: Find $(\bu_h,p_h,\theta_h)\in \mathbf{H}_h+u_{0,I}\times Q_h\times M_h+\theta_{0,I}$ such that
\begin{equation}
	\label{eq:cross-flow-fv1-nitsche}
	\begin{aligned}
		A((\bu_h,p_h),(\bv_h,q_h))+ \widetilde{a}(\bu_h,(\bu_h,\bv_h)) +B(\theta_h,(\bv_h,q_h))&=F(\bv_h,q_h),\\
		d(\theta_h,\tau_h) + \widetilde{c}(\tau_h,(\bu_h,\theta_h))&=0,
	\end{aligned}
\end{equation}
for all $(\bv_h,q_h)\in \mathbf{H}_h\times Q_h$ and for all $\forall \tau_h\in M_h$, where $u_{0,I}$ and $\theta_{0,I}$ are suitable interpolants of $u_0$ and $\theta_0$.

Below we prove that \eqref{eq:cross-flow-fv1-nitsche} contains all the elements to be consistent.
\begin{lemma}
	If $(\bu,p,\theta)$ is the solution of \eqref{model-eq1}-\eqref{model-eq3} together with boundary conditions \eqref{model-eq6}-\eqref{model-eq9}, then $(\bu,p,\theta)$ solves \eqref{eq:cross-flow-fv1-nitsche}. 
\end{lemma}
\begin{proof}
	We start by multiplying \eqref{model-eq9} by $\bv_h\cdot\bn$, with $\bv_h\in \mathbf{H}_h$, integrate over $\Gamma_m$, and multiply by $\alpha h^{-1}$ in order to have
	\begin{equation}
		\label{eq:lemma-consistencia-1}
		\alpha \int_{\Gamma_m}h^{-1}(\bu\cdot\bn)(\bv_h\cdot\bn)=\alpha \int_{\Gamma_m}h^{-1}\left(\frac{\Delta P - \kappa \theta}{I_0}\right)(\bv_h\cdot \bn).
	\end{equation}
	Similarly, we now multiply $\eqref{model-eq9}$ by $\bn^t\left(\mu_0\frac{\partial\bv_h}{\partial\bn}-q_h\bn\right)$, where $q_h\in Q_h$, to obtain
	\begin{equation}
		\label{eq:lemma-consistencia-2}
		\int_{\Gamma_m}\bn^t\left(\mu_0\frac{\partial\bv_h}{\partial\bn}-q_h\bn\right)(\bu\cdot\bn)=\int_{\Gamma_m}\bn^t\left(\mu_0\frac{\partial\bv_h}{\partial\bn}-q_h\bn\right)\left(\frac{\Delta P - \kappa \theta}{I_0}\right).
	\end{equation}
	Note that multiplying \eqref{model-eq1}-\eqref{model-eq3} with $(\bv_h,q_h,\tau_h)\in \mathbf{H}_h\times Q_h \times M_h $, integrating by parts and using boundary conditions \eqref{model-eq6}-\eqref{model-eq8} we have that
	\begin{equation}
		\label{eq:lema-consistencia}
		\begin{aligned}
			a(\bu,\bv_h)+b(\bv_h,p) + \widetilde{a}(\bu,(\bu,\bv_h)) - \int_{\Gamma_m}\bn^t\left(\mu_0\frac{\partial\bu}{\partial\bn}-p\bn\right)(\bv_h\cdot\bn)&=0,\\
			b(\bu,q_h)&=0,\\
			d(\theta,\tau_h) + \widetilde{c}(\bu,(\theta,\tau_h))&=0.
		\end{aligned}
	\end{equation}
	Then, by adding \eqref{eq:lemma-consistencia-1}, \eqref{eq:lemma-consistencia-2} and \eqref{eq:lema-consistencia} we obtain that
	$$
		\begin{aligned}
			A((\bu,p),(\bv_h,q_h))+ \widetilde{a}(\bu,(\bu,\bv_h)) +B(\theta,(\bv_h,q_h))&=F(\bv_h,q_h),\\
			d(\theta,\tau_h) + \widetilde{c}(\tau_h,(\bu,\theta))&=0,
		\end{aligned}
	$$
	for all $(\bv_h,q_h)\in \mathbf{H}_h\times Q_h$ and for all $\forall \tau_h\in M_h$.
	The proof is complete.
\end{proof}
The above lemma gives the modified variational formulation using the Nitsche method, and its consistency. Note that the nonlinearity is present in the convective terms and the boundary term $\int_{\Gamma_m\cup\Gamma_s}\theta
(\bu\cdot\bn)\tau$, where the nonlinearity in the equations is caused only by the velocity. Hence, the system \eqref{eq:cross-flow-fv1-nitsche} can be rearranged such that a fixed point iteration technique is straightforward.

%

\section{Numerical experiments: validation and discussion}
In this section we will show some numerical results in different channel configurations. The implementation of the scheme is perfomed in FEniCS, and the meshes were created using Gmsh. The numerical experiments have been developed with Taylor-Hood elements. However, similar results were obtained using the mini-element. For the concentration case, it was sufficient to consider the lowest case with $\ell=1$. All the cases consider the Nitsche constant $\alpha=1$.

We consider the geometrical parameters of a common membrane channel ``unit" whose length is defined by a subsection of the channel that allows a fully developed flow \cite{luo2020hybrid}. The channel thickness and spacer diameter is that of a module with standard $28\,mil$ gap spacer mesh \cite{kucera2015reverse}. For the channel width $W$ (perpendicular to the simulation domain) we will consider $W=L$. More precisely, we have
$$
L=1.5\cdot 10^{-2}\,m,\quad d=7.4\cdot 10^{-4}\,m,\qquad W=1.5\cdot 10^{-2}\,m,\qquad d_S=3.6\cdot 10^{-4}\,m,
$$
while the physical parameters \cite{bernales2017prandtl,nayar2016thermophysical} are taken as
$$
\begin{aligned}
	&\kappa=4955.144\,Pa, \quad I_0=8.41\cdot10^{10} Pa\,\cdot s/m,\quad \mu_0=8.9\cdot10^{-4}kg\,m^{-1}s^{-1},\\
	&\rho_0=1027.2\,kg/m^3,\quad D_0=1.5\cdot10^{-9} m^2/s.
\end{aligned}
$$
With respect to the boundary data, we will consider combinations among the following:
$$
\begin{aligned}
	&u_0= 1.29\times k\cdot10^{-1}m/s, \quad k=0.5,1,2.\\
	&C_0= 600 \text{ mol }m^{3}.\\
	&\Delta P = 4053000\,Pa\text{ or }5572875\,Pa.
\end{aligned}
$$
As suggested by several studies, all the experiments use a highly refined mesh close to the membrane $\Gamma_m$.

\subsection{Validation test}
This experiment aims to test the convergence of our method. To this end, we consider a channel with no spacers, i.e., we set $\Gamma_s=\emptyset$. The domain $\Omega=(0,L)\times(0,d)$ is meshed with uniform elements such that $n_x=(L/d)n_y$, where $n_x$ and $n_y$ denote the number of cells in the $x$ and $y$ directions, respectively. We divide this test in two cases, that we detail below.

\subsubsection{Case 1} The first case consists in validating the method by comparing the pressure drop with those of classical analytical models of momentum, namely, Poiseuille and Berman flow models. These two models have the advantage that the pressure drop, denoted by $\Delta p(x,d/2):=p(0,d/2)-p(x,d/2)$, can be obtained by solving the equations of motion with boundary conditions
$$
\bu\cdot \bn=\frac{\Delta P}{I_0},\quad\bu\cdot \bn=0, 
$$
on $\Gamma_m$, for Berman and Poiseuille flow, respectively. Note that, for the Poiseuille flow, we have a no-slip and no-penetration boundary condition (Dirichlet boundary condition) that can be imposed strongly on the space. Hence, we have the following pressure drop equation for a constant permeable wall
$$
	\Delta p(x,d/2):=\left(\frac{1}{2}\rho_0u_0^2\right)\left(\frac{24}{Re}-\frac{648}{35}\frac{Re_{\bn}}{Re}\right)\left(1-\frac{2Re_{\bn}}{Re}\frac{x}{\widetilde{d}}\right)\left(\frac{x}{\widetilde{d}}\right),
$$
where a constant permeate velocity $\bu\cdot\bn$ is assumed. On the other hand, the case of impermeable walls assumes that $Re_{\bn}=0$, so the pressure drop equation is reduced to
$$\Delta p(x,d/2):=\left(\frac{1}{2}\rho_0u_0^2\right)\left(\frac{24}{Re}\right)\left(\frac{x}{\widetilde{d}}\right).$$
In the above equations, $Re:=\frac{4\rho_0\widetilde{d}u_0}{\mu_0}$ is the cross flow Reynolds number, whereas $Re_{\bn}=\frac{\rho_0\tilde{d}(\bu\cdot\bn) }{\mu_0}$ is the Reynolds number at the channel walls. The parameter $\widetilde{d}$ denotes half of the channel height ($\widetilde{d}=d/2$). 

In Figure \ref{fig:comp1} we observe the exact pressure drop compared with our computed pressure drop, denoted by $\Delta p_h(x,d/2)$. To compare with our approximate solution, we have selected $10$ equally spaced pressure drop values in the channel. We note that the proposed method agrees perfectly with the analytical axial pressure drop for both, the permeable and impermeable channel. Also, a similar pressure drop was observed for $\Delta P =4053000$ Pa and $\Delta P=5572875$ Pa since they correspond to $\bu\cdot\bn\approx4.819\cdot 10^{-5}$ m/s and $\bu\cdot\bn\approx6.626\cdot10^{-5}$ m/s, respectively. Therefore, $\Delta P$ has little impact on the pressure drop on the channel.

\begin{figure}[!ht]
\centering
\includegraphics[scale=0.25]{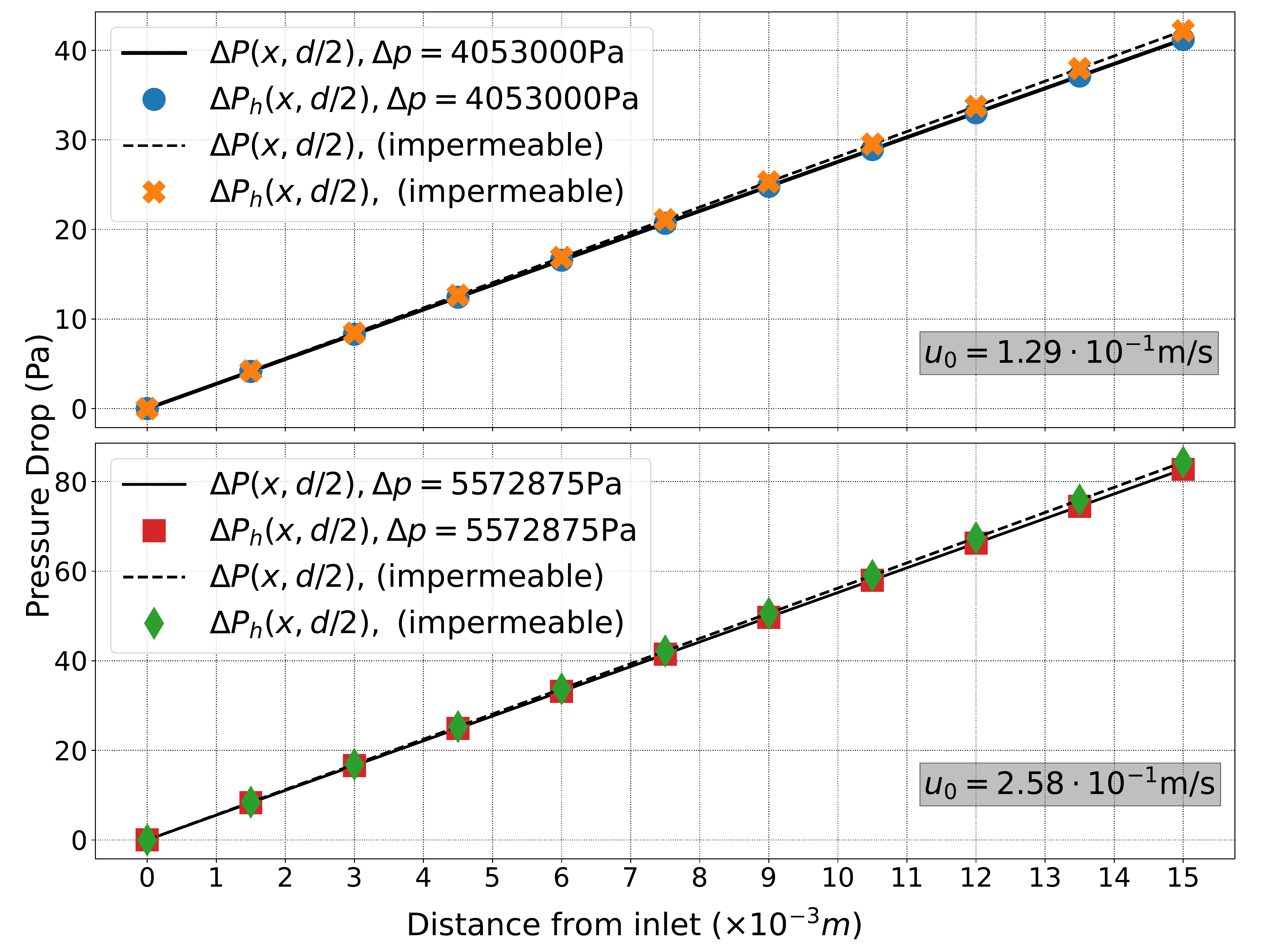}
\caption{Test 1. Comparison of the exact and approximate pressure drop.}
\label{fig:comp1}
\end{figure}
\subsubsection{Case 2} We now aim to test the method when the non-linear boundary condition \eqref{model-eq9} is considered. To this end, we compute the total flow (see equation \eqref{eq:totmflow}) over the membrane channel in order to observe the stability of the result for different number of elements:
\begin{equation}
\label{eq:totmflow}
\dot{m}_{tot}=\rho W\int_{\Gamma_{m}}\abs{\boldsymbol{u}_{y}(s)}ds.
\end{equation}
The importance of this test lies in the determination of the number of elements necessary  to obtain accurate results. Some studies suggest that refinement towards the membrane are required for better results \cite{li2016three,luo2020hybrid}. A total of four refinement near $\Gamma_m$ have been used. We set $N=2j,\, j=1,\ldots,19,$ for uniform refinement, whereas $N=5j$ is used for the meshing towards $\Gamma_m$.

In Figure \ref{fig:comp2} we observe the average flux for several inlet velocities and osmotic pressure when using uniform refinements. In both cases, it is observed that from $5\cdot10^3$ elements the total mass flow is stable. All the simulations in this work were tested in a similar manner, preferring the meshing towards $\Gamma_m$ with a variable amount of elements, but always above $10^4$ cells.
\begin{figure}[!ht]
	\centering
	\includegraphics[scale=0.28]{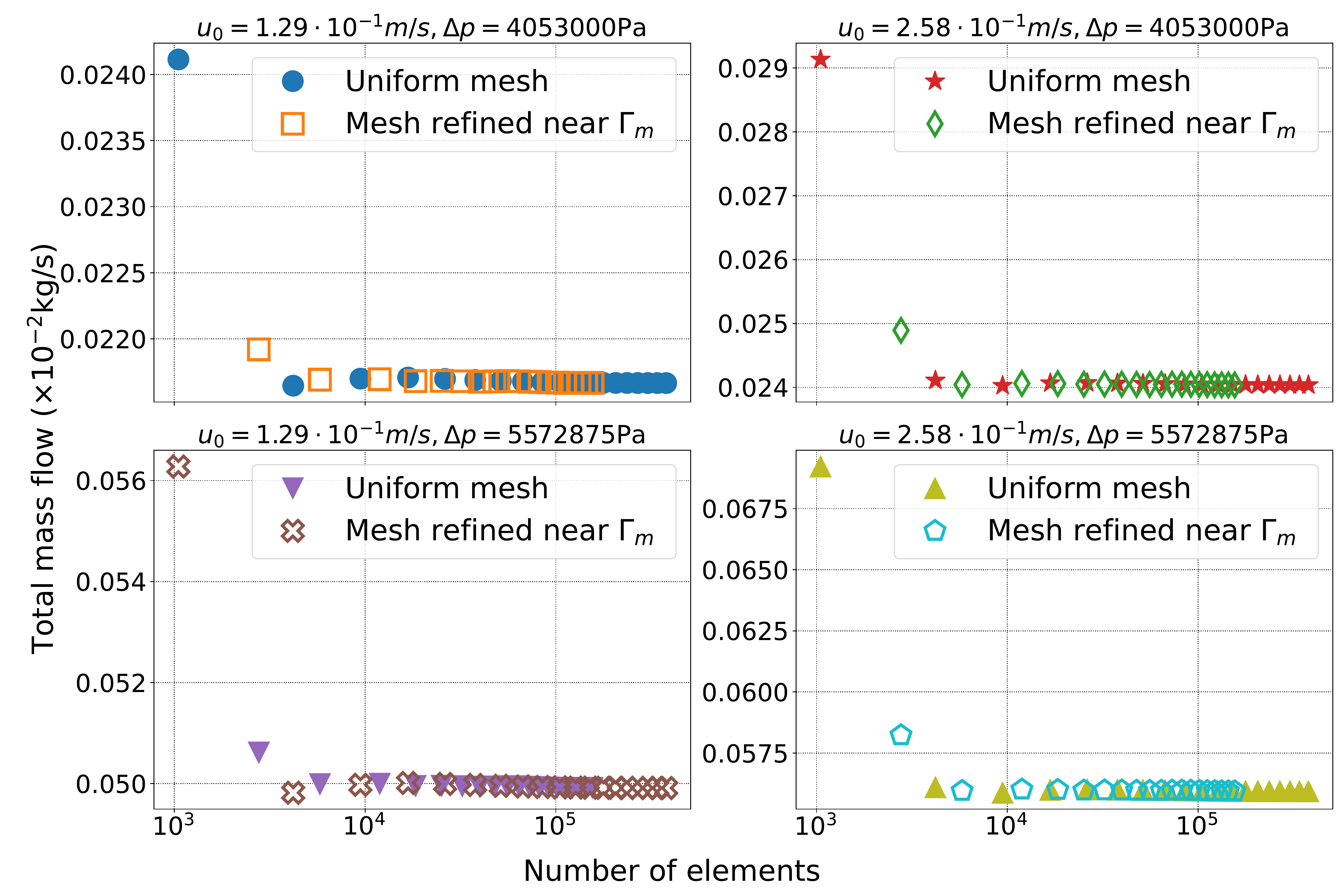}
	\caption{Test 1. Comparison of the stability in the total mass flow computation stability when using uniform refinement and meshes refined towards $\Gamma_m$.}
	\label{fig:comp2}
\end{figure}
\subsection{Test 2. Channel with different spacer configurations}
In this section we will depict the velocity fields, pressure distribution and concentration polarization along a channel with different spacer configurations. Discussion of the simulations are presented in Section \ref{sec:discusion} below.
\subsubsection{Spacers in cavity configuration}
The numerical results for the spacers in cavity configuration are presented in Figures \ref{fig:cavity1}-\ref{fig:cavity3}. 
\begin{figure}[!ht]
	\centering
		\begin{minipage}{1\linewidth}
		\centering
		\footnotesize{$u_0=6.45\cdot10^{-2} m/s,\Delta P =4053000Pa $}
		\includegraphics[trim= 0cm 35cm 0cm 27cm,clip, scale=0.1]{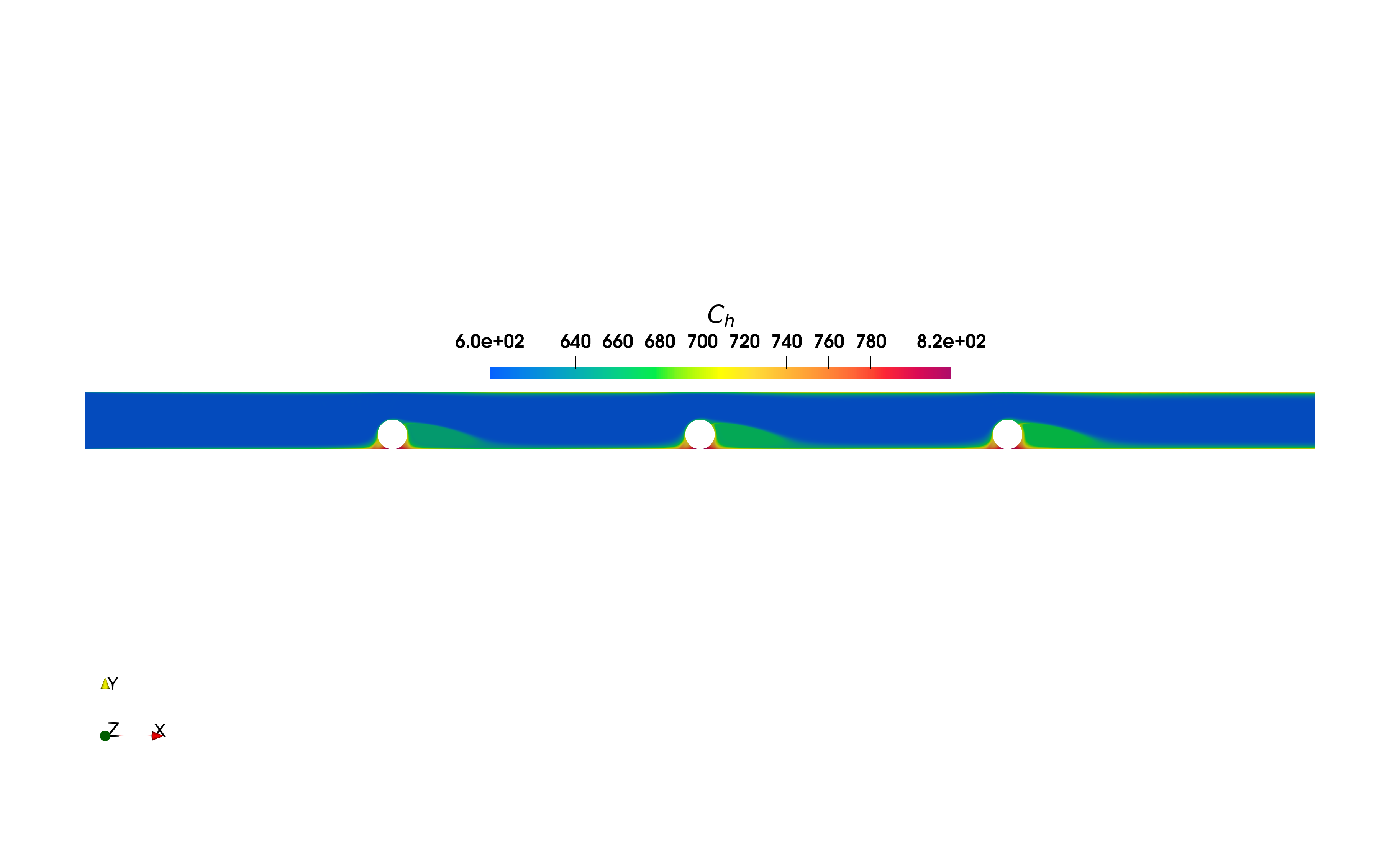}
		
	\end{minipage}
	\begin{minipage}{1\linewidth}
		\centering
		\footnotesize{$u_0=2.58\cdot10^{-1}m/s,\Delta P =4053000Pa $}
		\includegraphics[trim= 0cm 35cm 0cm 27cm,clip, scale=0.1]{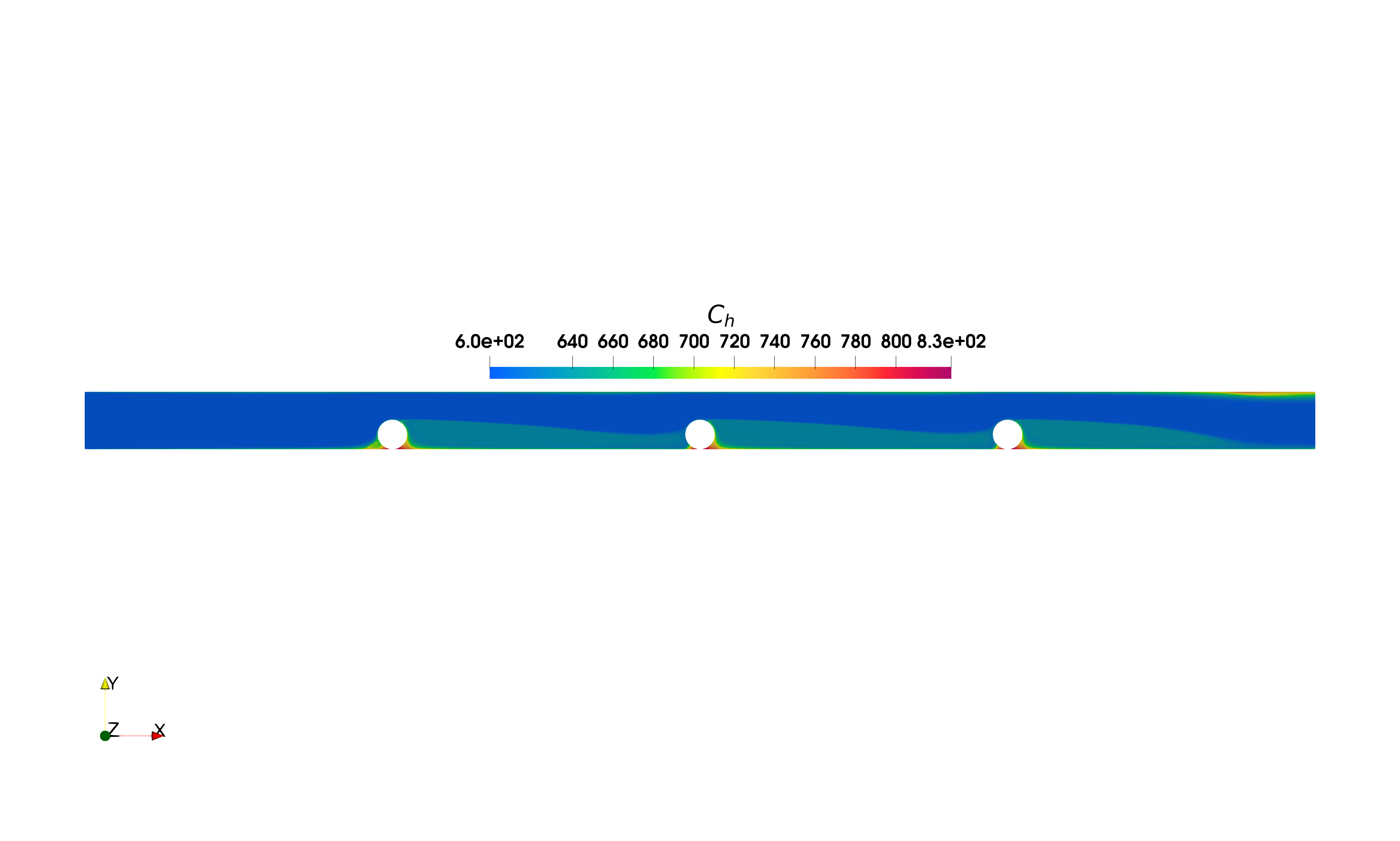}
		
	\end{minipage}
	\begin{minipage}{1\linewidth}
		\centering
		\footnotesize{$u_0=6.45\cdot10^{-2} m/s,\Delta P =5572875Pa $}
		\includegraphics[trim= 0cm 35cm 0cm 27cm,clip, scale=0.1]{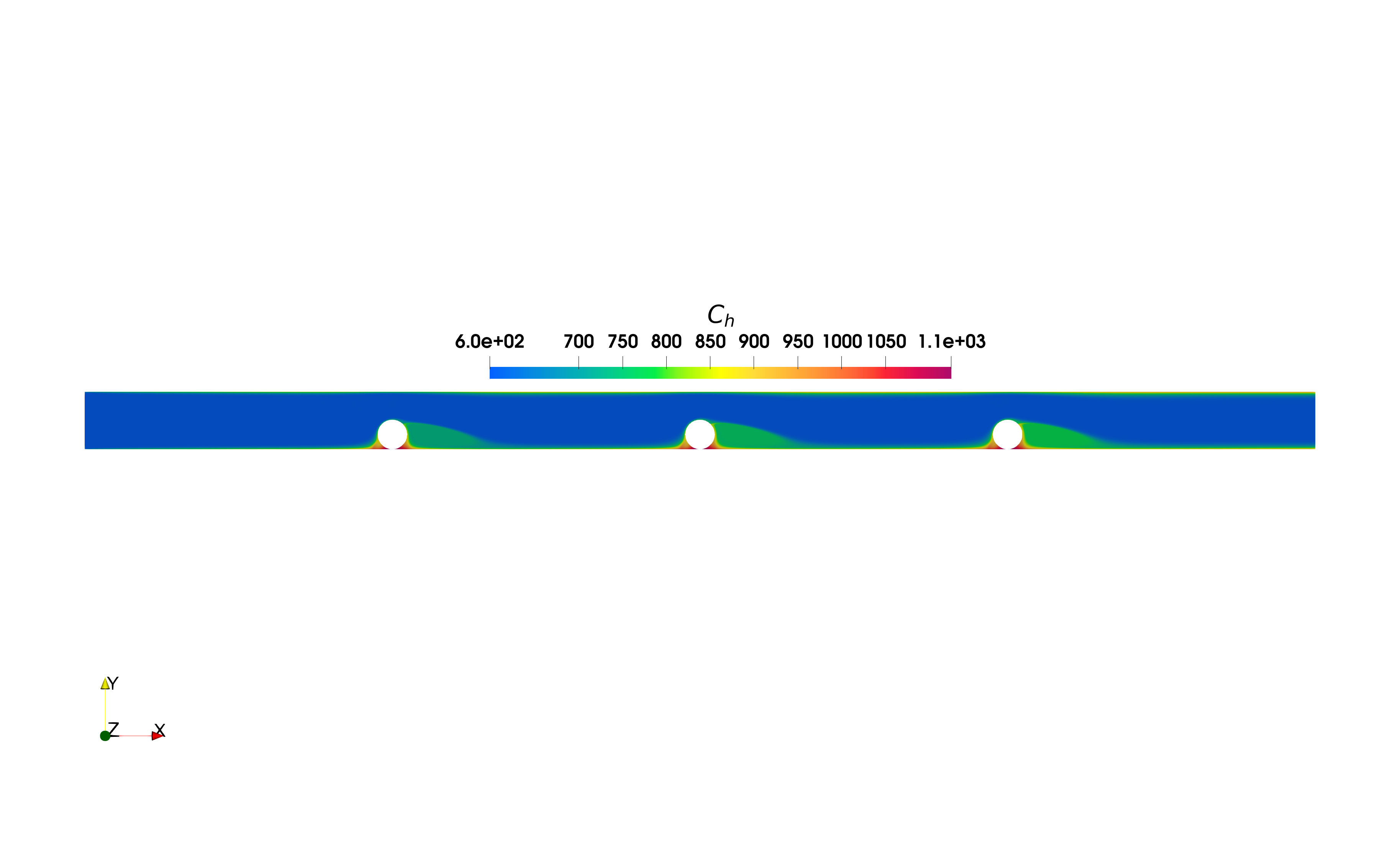}
		
	\end{minipage}
	\begin{minipage}{1\linewidth}
		\centering
		\footnotesize{$u_0=2.58\cdot10^{-1}m/s,\Delta P =5572875Pa $}
		\includegraphics[trim= 0cm 35cm 0cm 27cm,clip, scale=0.1]{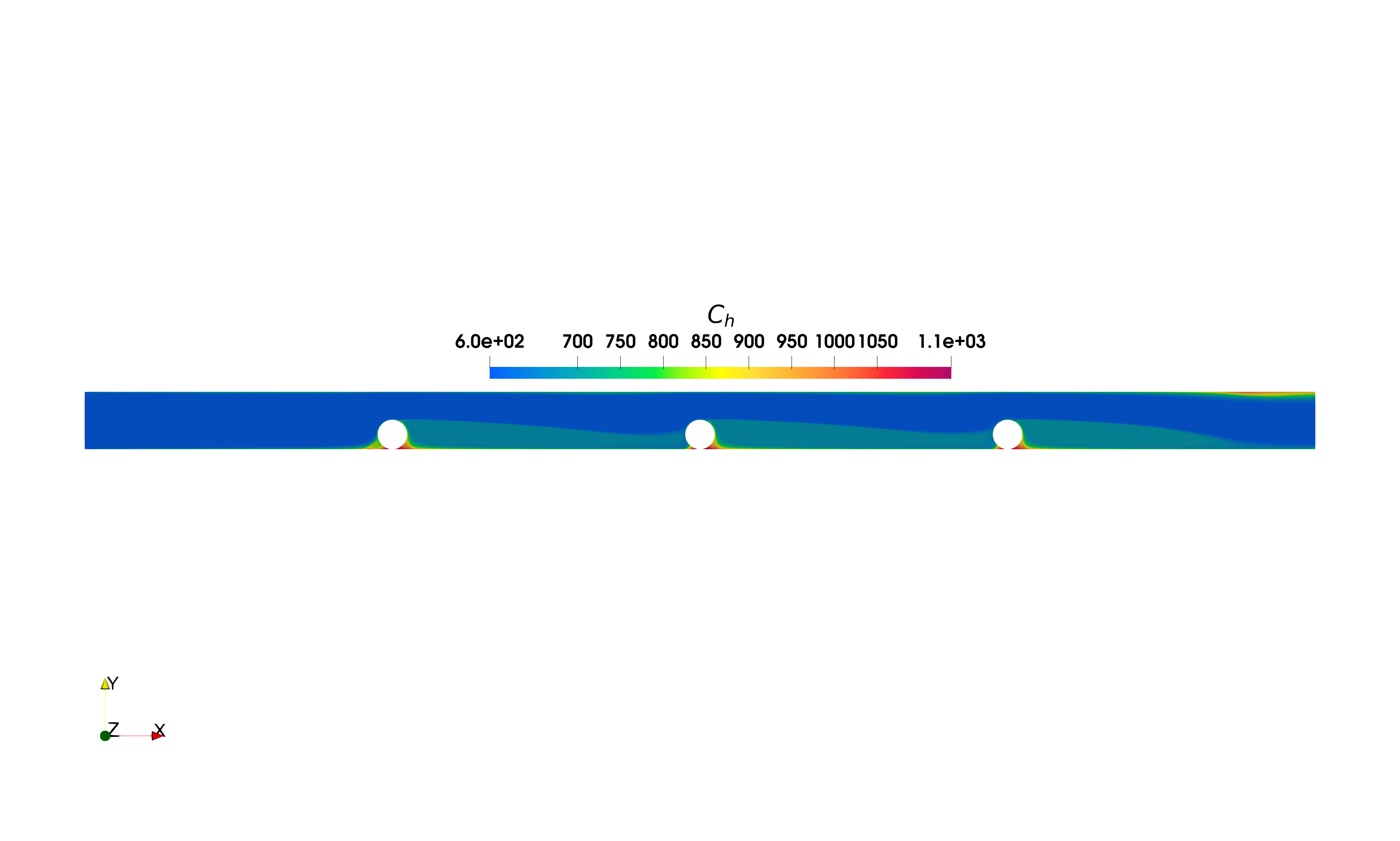}
		
	\end{minipage}
\caption{Test 2. Comparison of concentration levels using a cavity spacers configuration.}
\label{fig:cavity1}
\end{figure}

\begin{figure}[!ht]
	\centering
	\footnotesize{$u_0=6.45\cdot10^{-2} m/s,\Delta P =4053000Pa $}
	\begin{minipage}{1\linewidth}
		\centering
		\includegraphics[trim= 2.5cm 8cm 1.5cm 10cm,clip, scale=0.48]{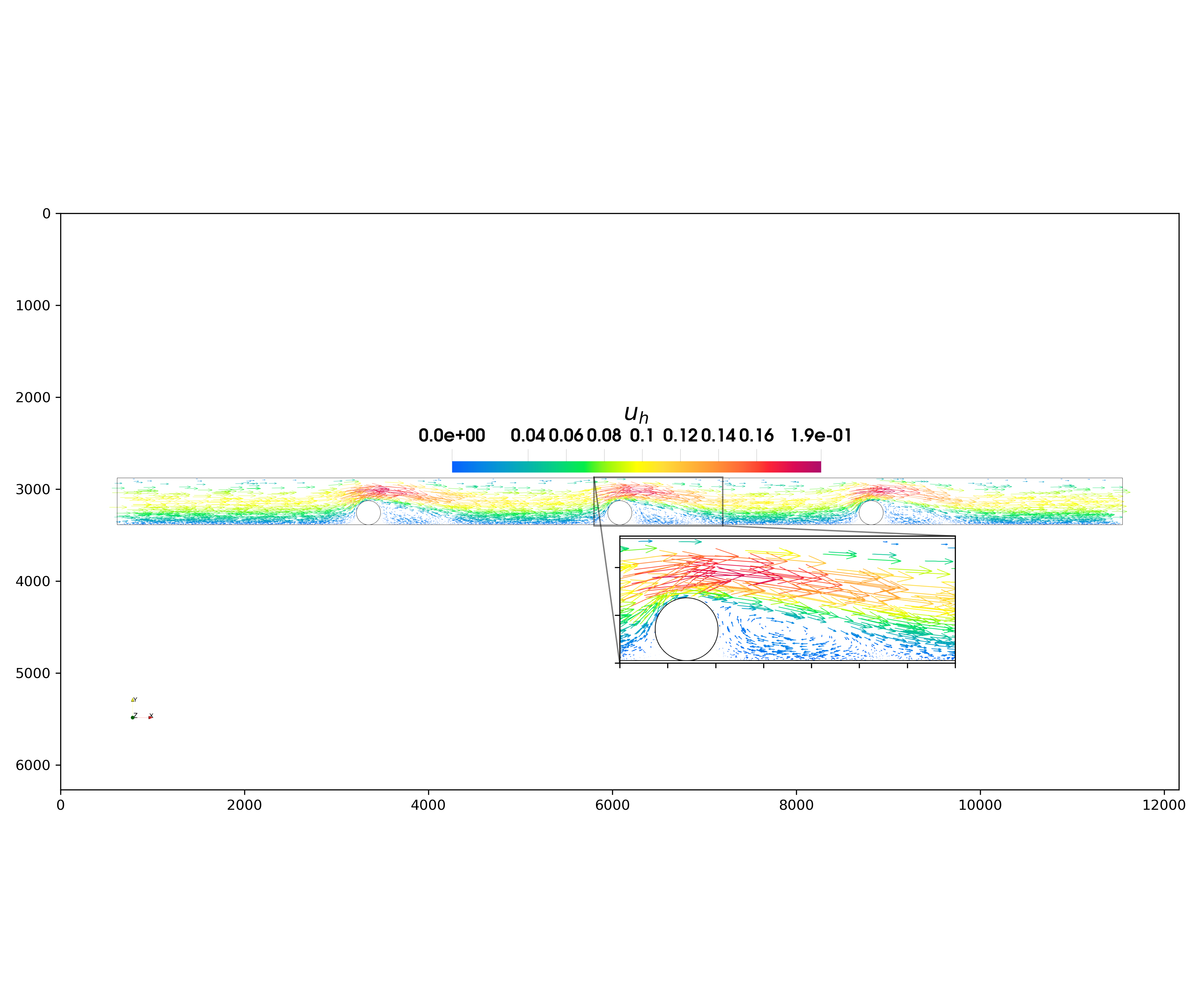}
		
	\end{minipage}
	\begin{minipage}{1\linewidth}
		\centering
		\footnotesize{$u_0=2.58\cdot10^{-1}m/s,\Delta P =4053000Pa $}
		\includegraphics[trim= 2.5cm 8cm 1.5cm 10cm,clip, scale=0.48]{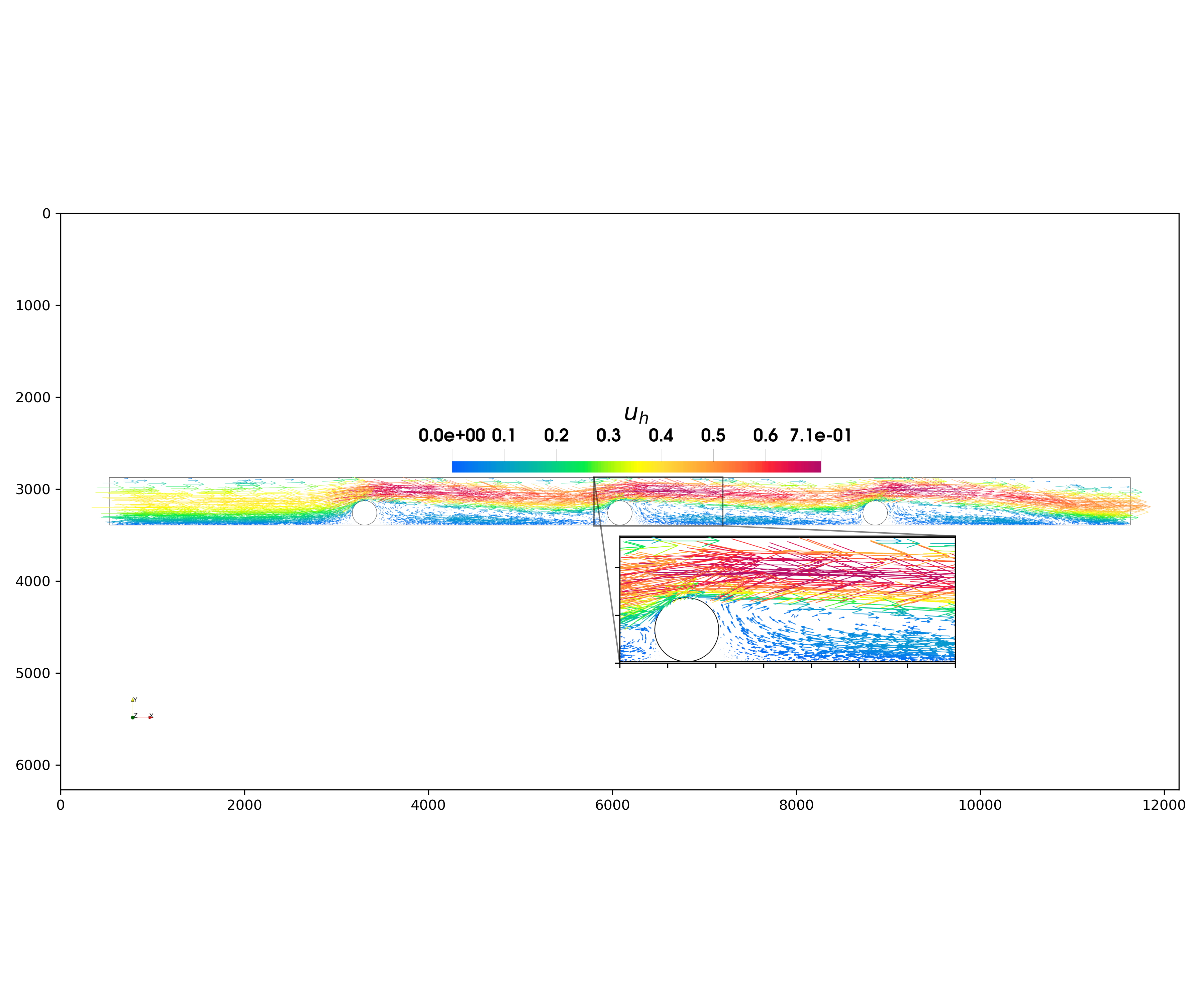}
	\end{minipage}
	\caption{Test 2. Comparison of velocity fields using a cavity spacers configuration.}
\label{fig:cavity2}
\end{figure}

\begin{figure}[!ht]
	\centering
	\begin{minipage}{1\linewidth}
		\centering
		\footnotesize{$u_0=6.45\cdot10^{-2} m/s,\Delta P =4053000Pa $}
		\includegraphics[trim= 0cm 35cm 0cm 27cm,clip, scale=0.1]{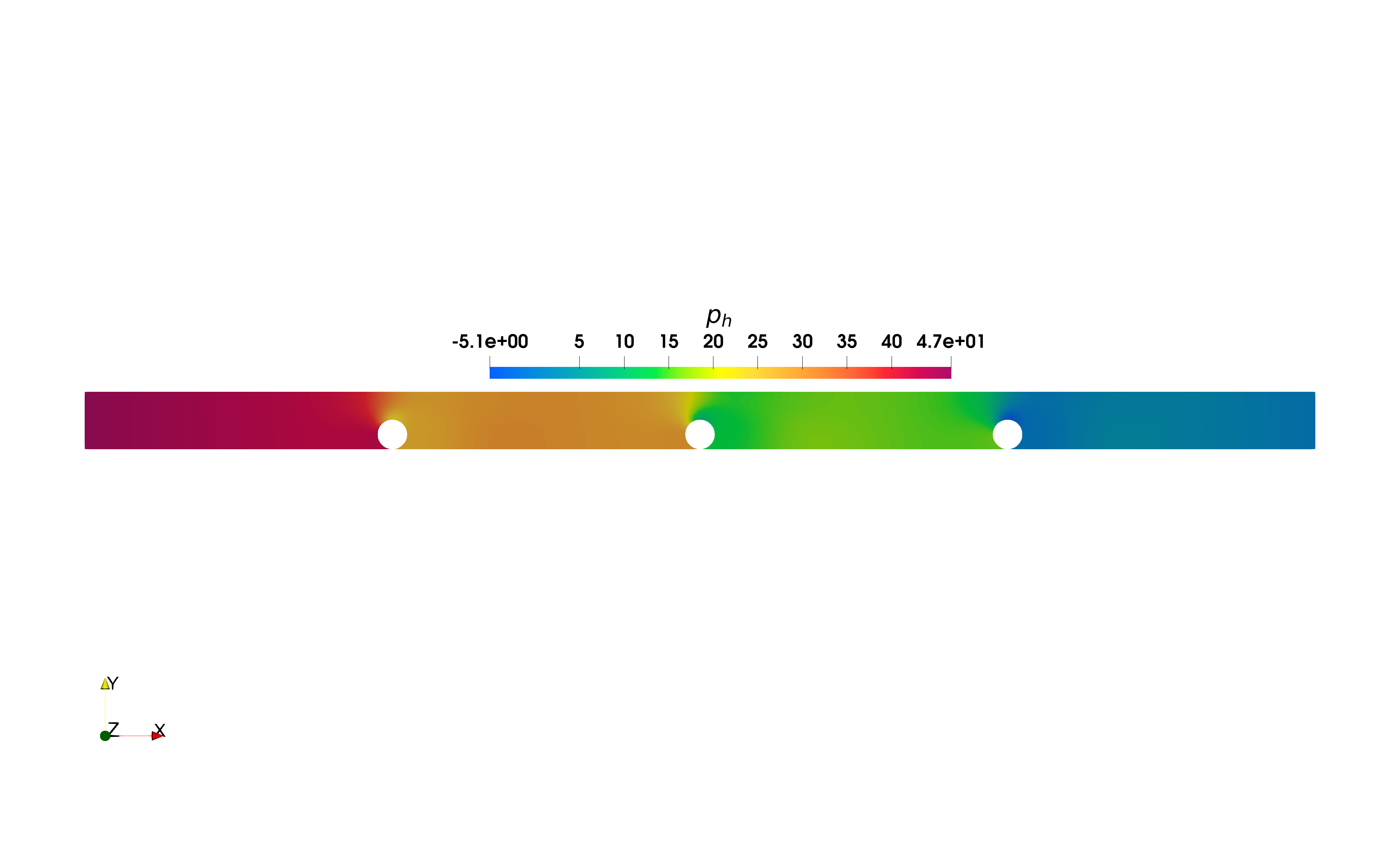}
		
	\end{minipage}
	\begin{minipage}{1\linewidth}
		\centering
		\footnotesize{$u_0=2.58\cdot10^{-1}m/s,\Delta P =4053000Pa $}
		\includegraphics[trim= 0cm 35cm 0cm 27cm,clip, scale=0.1]{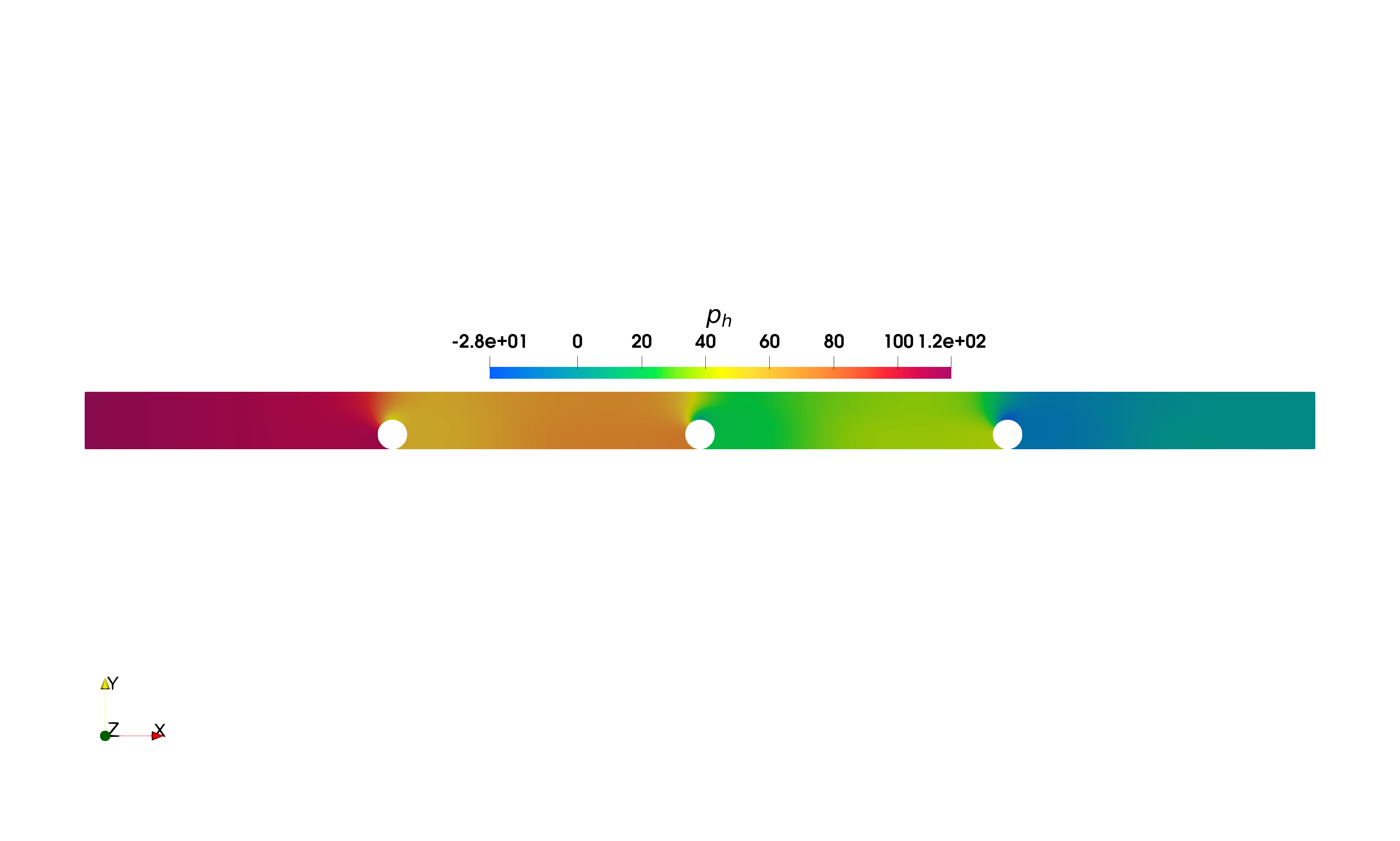}
		
	\end{minipage}
	\begin{minipage}{1\linewidth}
		\centering
		\footnotesize{$u_0=6.45\cdot10^{-2} m/s,\Delta P =5572875Pa $}
		\includegraphics[trim= 0cm 35cm 0cm 27cm,clip, scale=0.1]{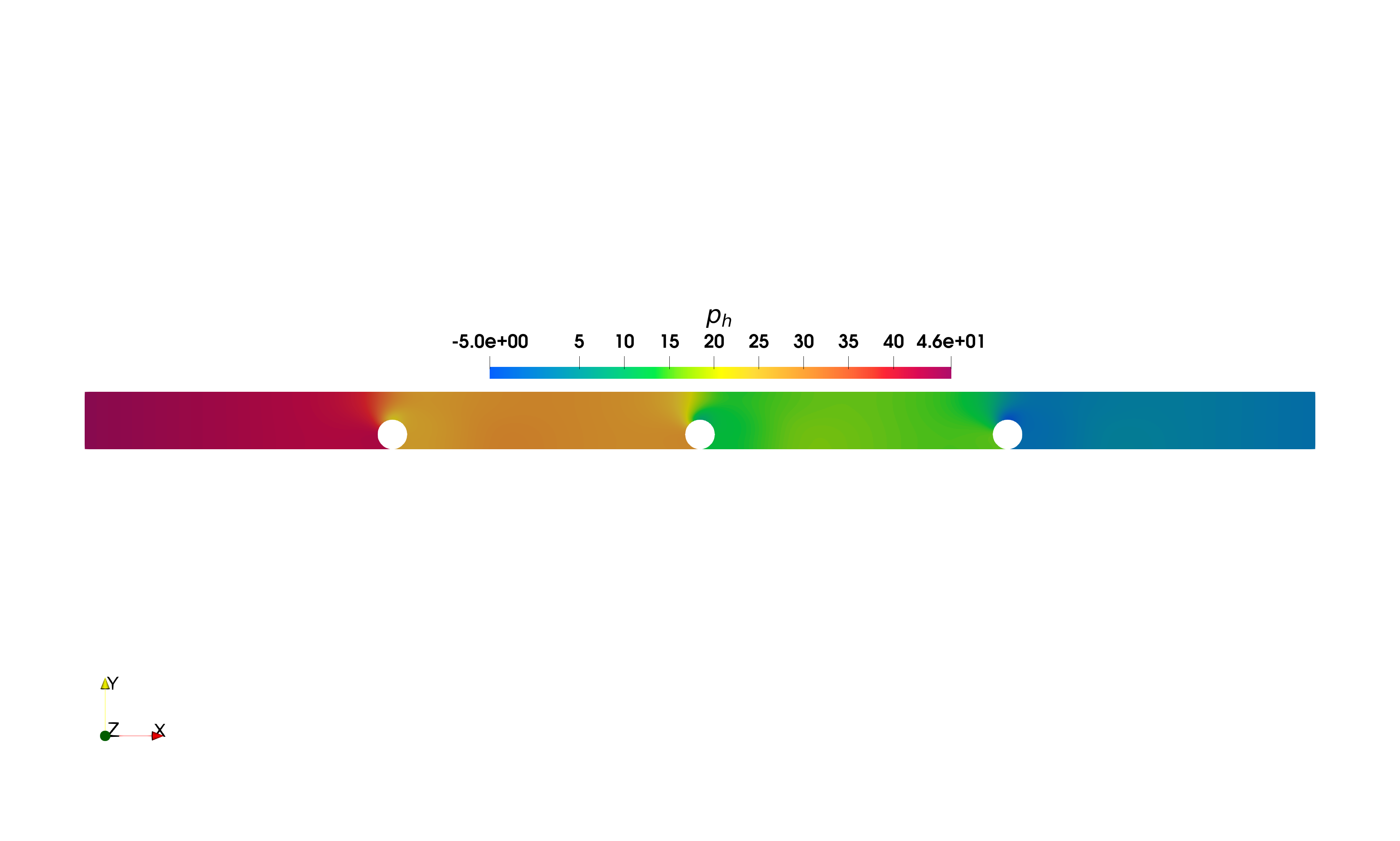}
		
	\end{minipage}
	\begin{minipage}{1\linewidth}
		\centering
		\footnotesize{$u_0=2.58\cdot10^{-1}m/s,\Delta P =5572875Pa $}
		\includegraphics[trim= 0cm 35cm 0cm 27cm,clip, scale=0.1]{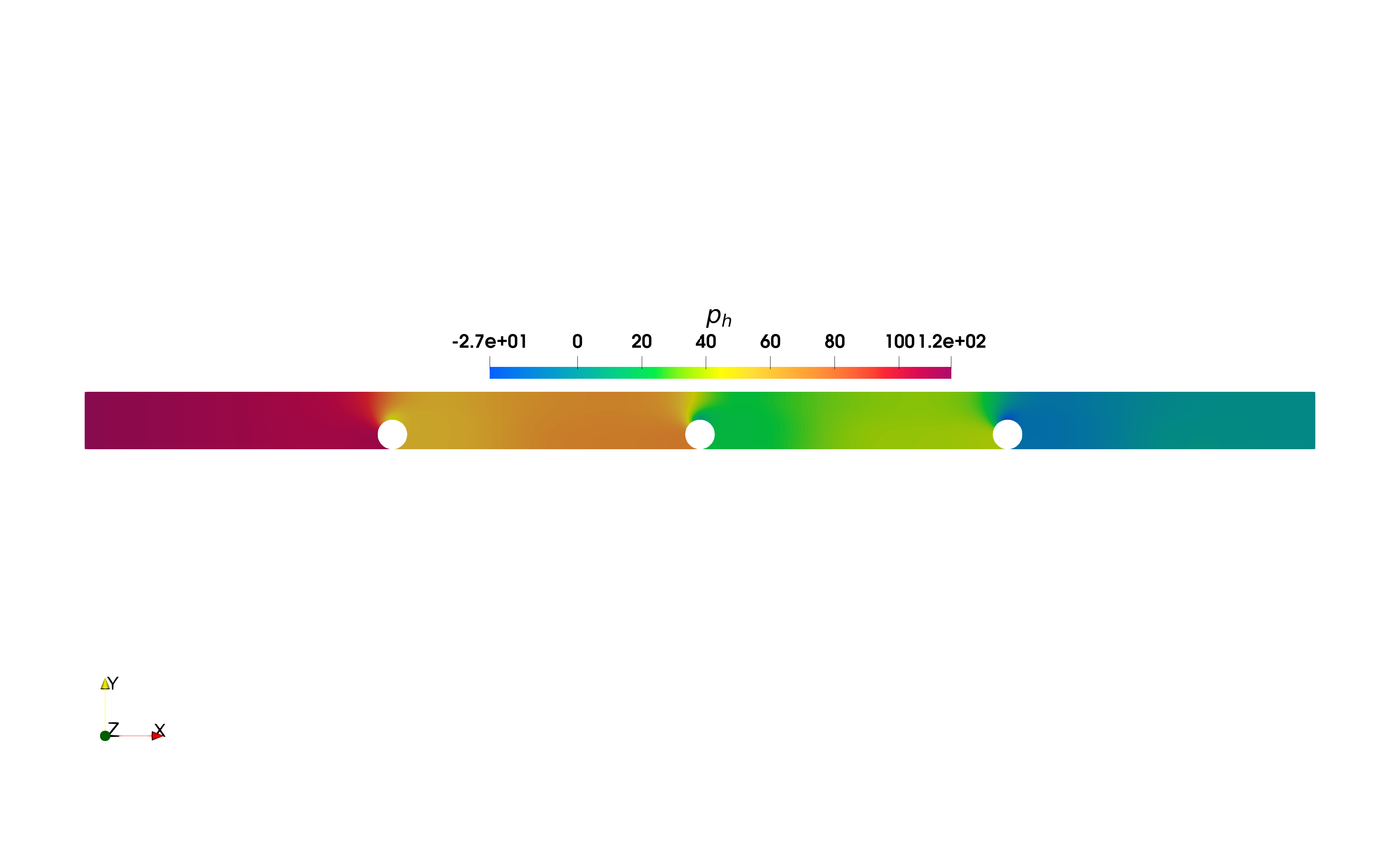}
		
	\end{minipage}
	\caption{Test 2. Comparison of relative pressure using a cavity spacers configuration.}
	\label{fig:cavity3}
\end{figure}
\newpage
\subsubsection{Zig-zag spacers configuration} 
The numerical results for the spacers in zig-zag configuration are presented in Figures \ref{fig:zigzag1}-\ref{fig:zigzag3}. 
\begin{figure}[!ht]
	\centering
	\begin{minipage}{1\linewidth}
		\centering
		\footnotesize{$u_0=6.45\cdot10^{-2} m/s,\Delta P =4053000Pa $}
		\includegraphics[trim= 0cm 35cm 0cm 27cm,clip, scale=0.1]{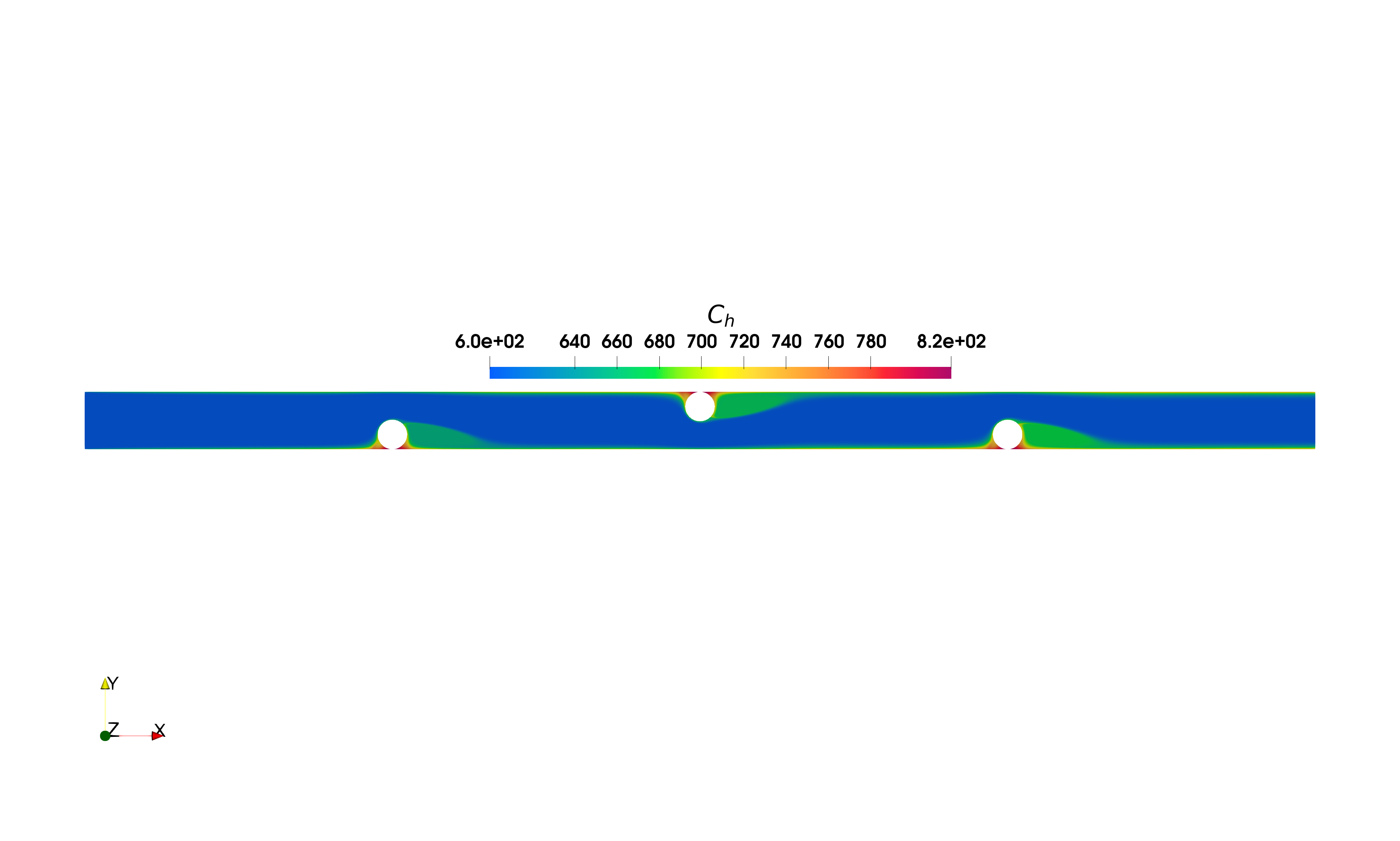}
		
	\end{minipage}
	\begin{minipage}{1\linewidth}
		\centering
		\footnotesize{$u_0=2.58\cdot10^{-1}m/s,\Delta P =4053000Pa $}
		\includegraphics[trim= 0cm 35cm 0cm 27cm,clip, scale=0.1]{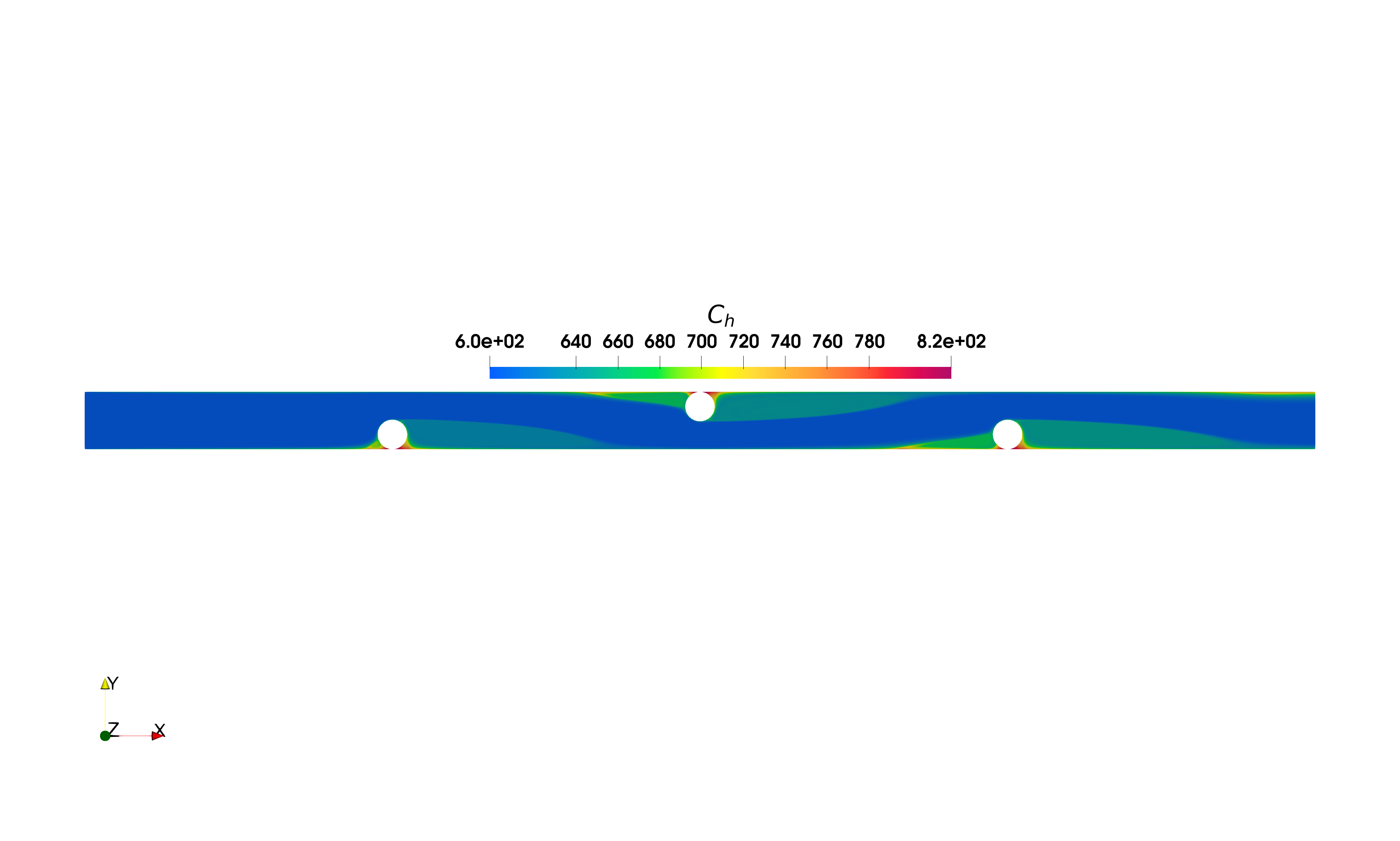}
		
	\end{minipage}
	\begin{minipage}{1\linewidth}
		\centering
		\footnotesize{$u_0=6.45\cdot10^{-2} m/s,\Delta P =5572875Pa $}
		\includegraphics[trim= 0cm 35cm 0cm 27cm,clip, scale=0.1]{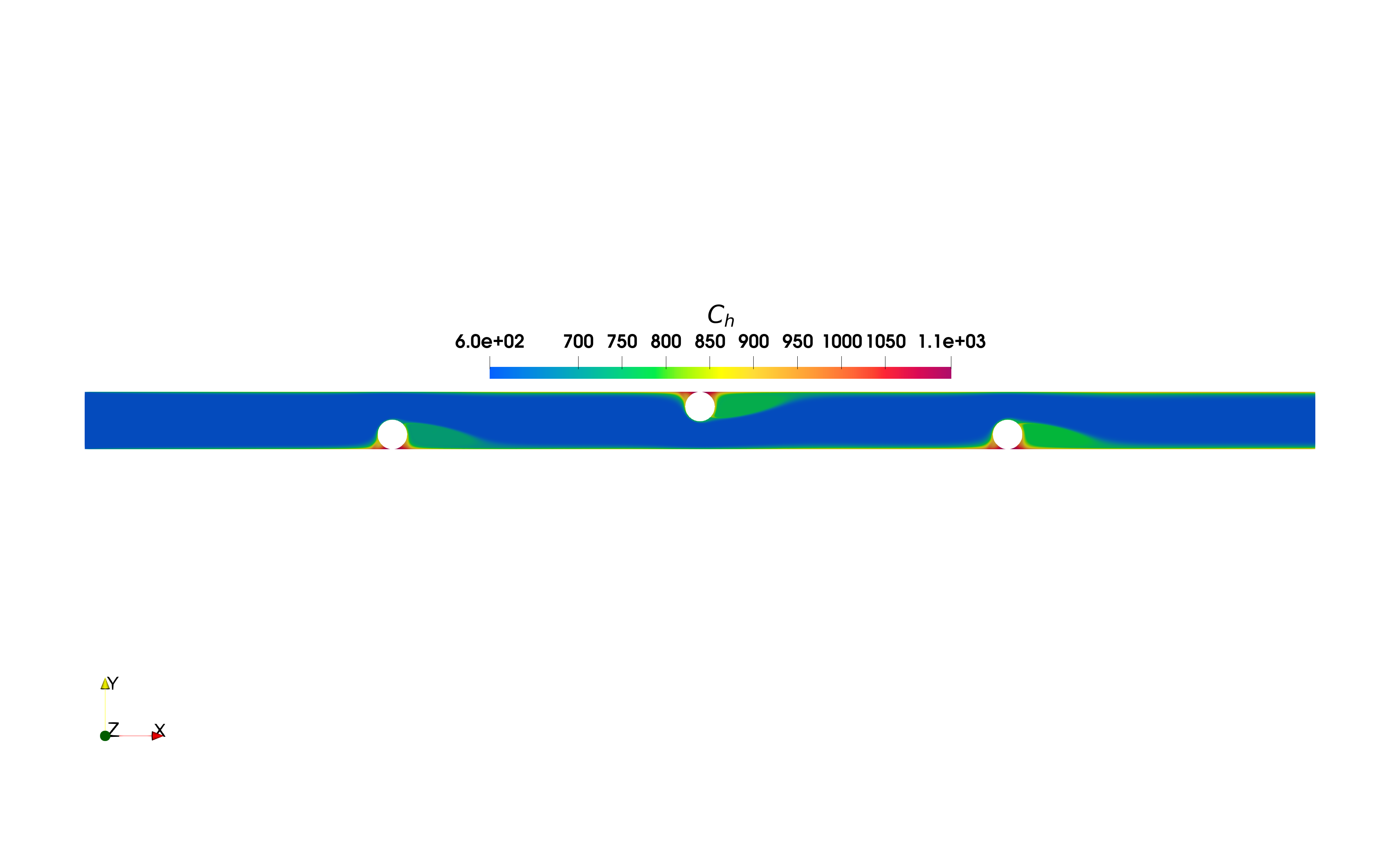}
		
	\end{minipage}
	\begin{minipage}{1\linewidth}
		\centering
		\footnotesize{$u_0=2.58\cdot10^{-1}m/s,\Delta P =5572875Pa $}
		\includegraphics[trim= 0cm 35cm 0cm 27cm,clip, scale=0.1]{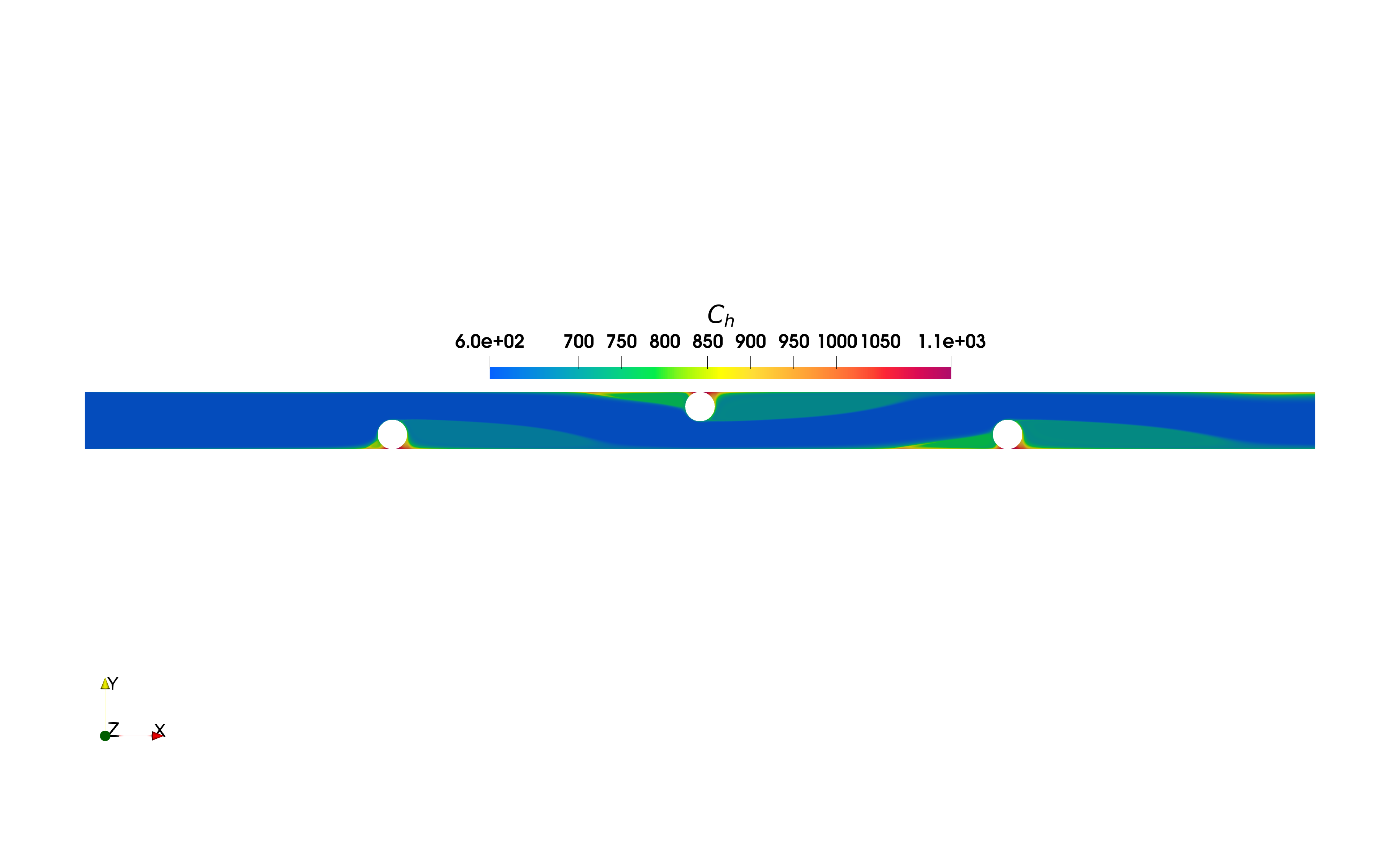}
		
	\end{minipage}
	\caption{Test 2. Comparison of concentration levels using a zig-zag spacers configuration.}
	\label{fig:zigzag1}
\end{figure}

\begin{figure}[!ht]
	\centering
	\begin{minipage}{1\linewidth}
		\centering
		\footnotesize{$u_0=6.45\cdot10^{-2} m/s,\Delta P =4053000Pa $}
		\includegraphics[trim= 2.5cm 8cm 1.5cm 10cm,clip, scale=0.48]{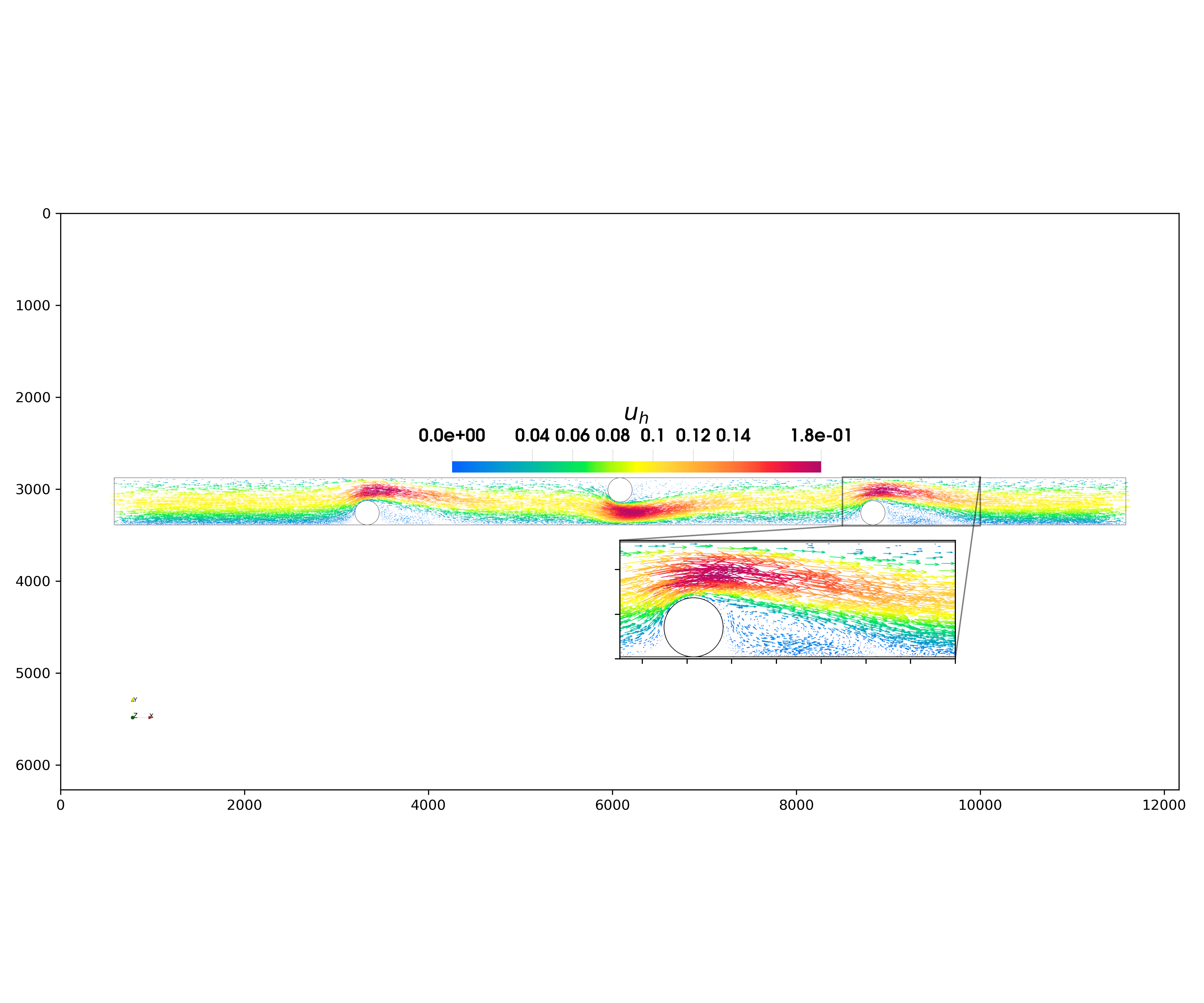}	
	\end{minipage}
	\begin{minipage}{1\linewidth}
		\centering
		\footnotesize{$u_0=2.58\cdot10^{-1}m/s,\Delta P =4053000Pa $}
		\includegraphics[trim= 2.5cm 8cm 1.5cm 10cm,clip, scale=0.48]{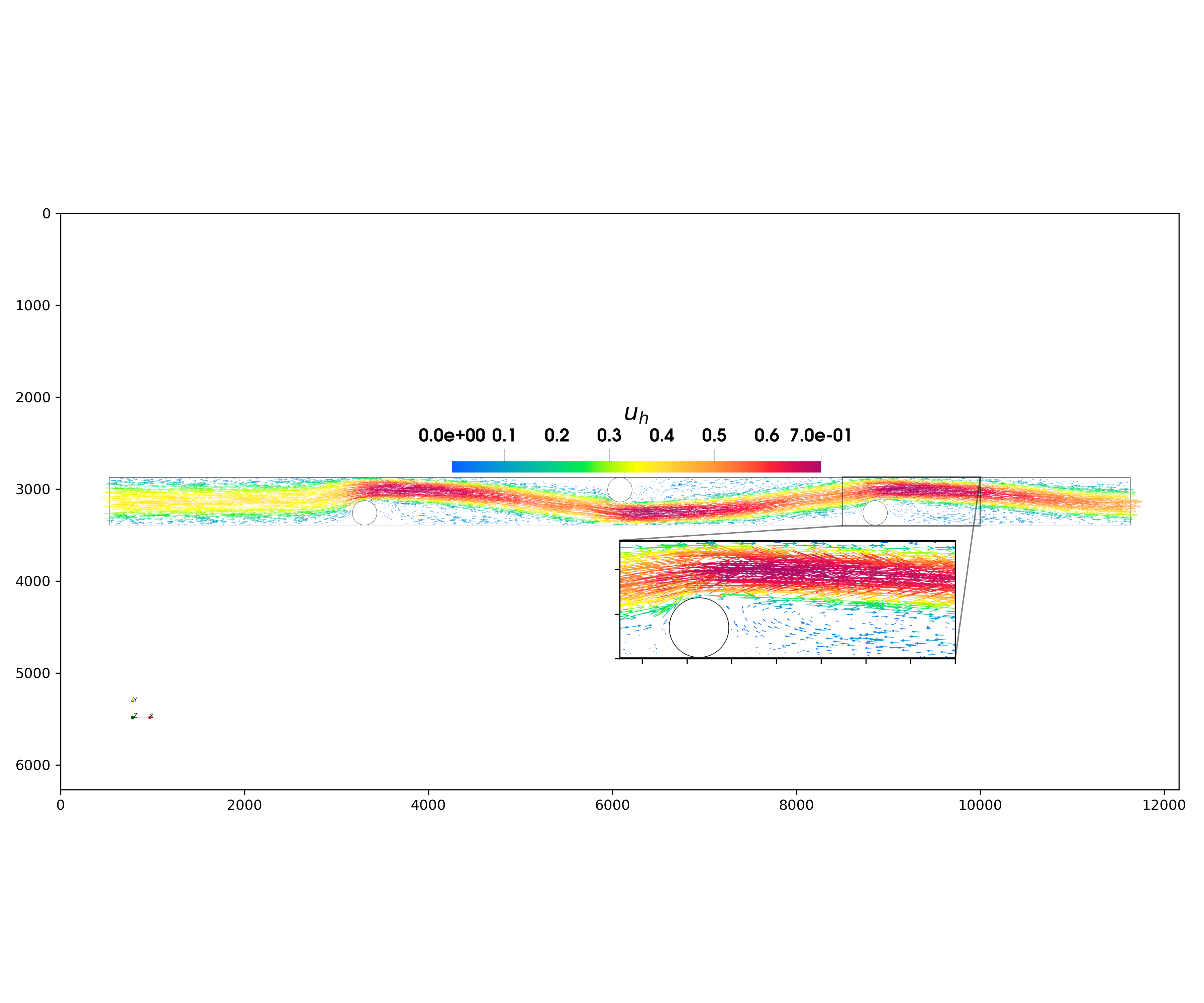}
	\end{minipage}
	\caption{Test 2. Comparison of velocity fields using a zig-zag spacers configuration.}
	\label{fig:zigzag2}
\end{figure}

\begin{figure}[!ht]
	\centering
	\begin{minipage}{1\linewidth}
		\centering
		\footnotesize{$u_0=6.45\cdot10^{-2} m/s,\Delta P =4053000Pa $}
		\includegraphics[trim= 0cm 35cm 0cm 27cm,clip, scale=0.1]{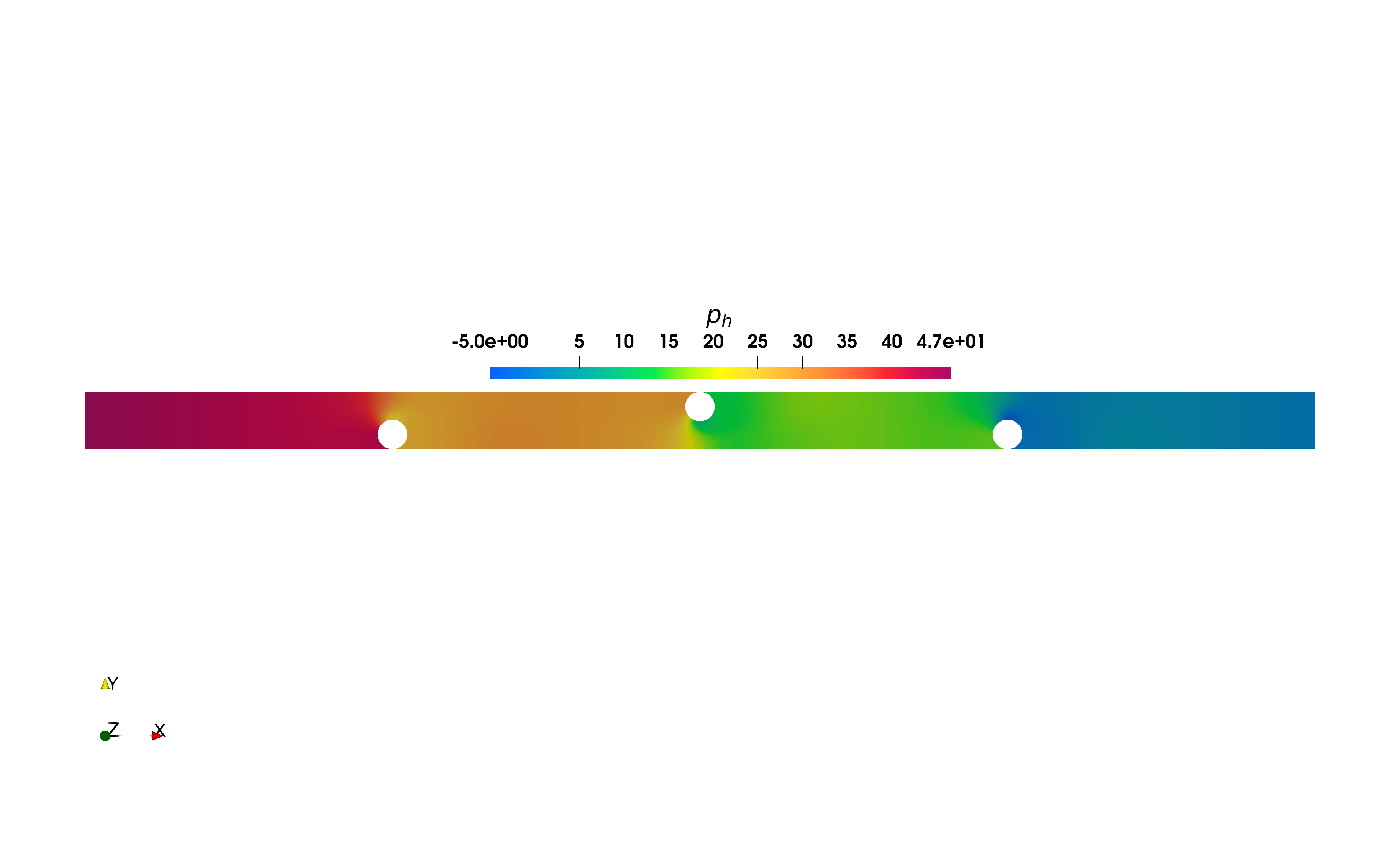}
		
	\end{minipage}
	\begin{minipage}{1\linewidth}
		\centering
		\footnotesize{$u_0=2.58\cdot10^{-1}m/s,\Delta P =4053000Pa $}
		\includegraphics[trim= 0cm 35cm 0cm 27cm,clip, scale=0.1]{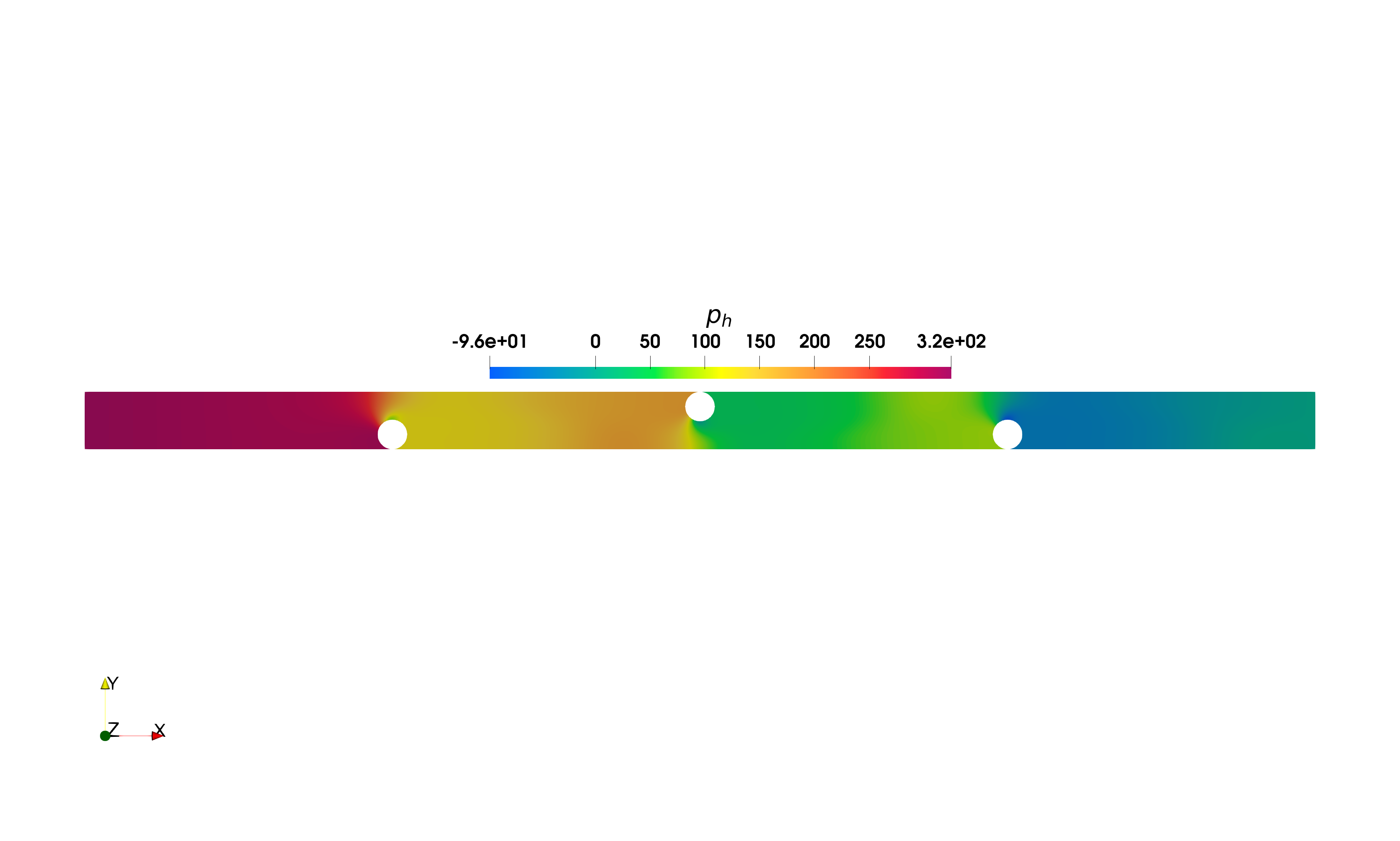}
		
	\end{minipage}
	\begin{minipage}{1\linewidth}
		\centering
		\footnotesize{$u_0=6.45\cdot10^{-2} m/s,\Delta P =5572875Pa $}
		\includegraphics[trim= 0cm 35cm 0cm 27cm,clip, scale=0.1]{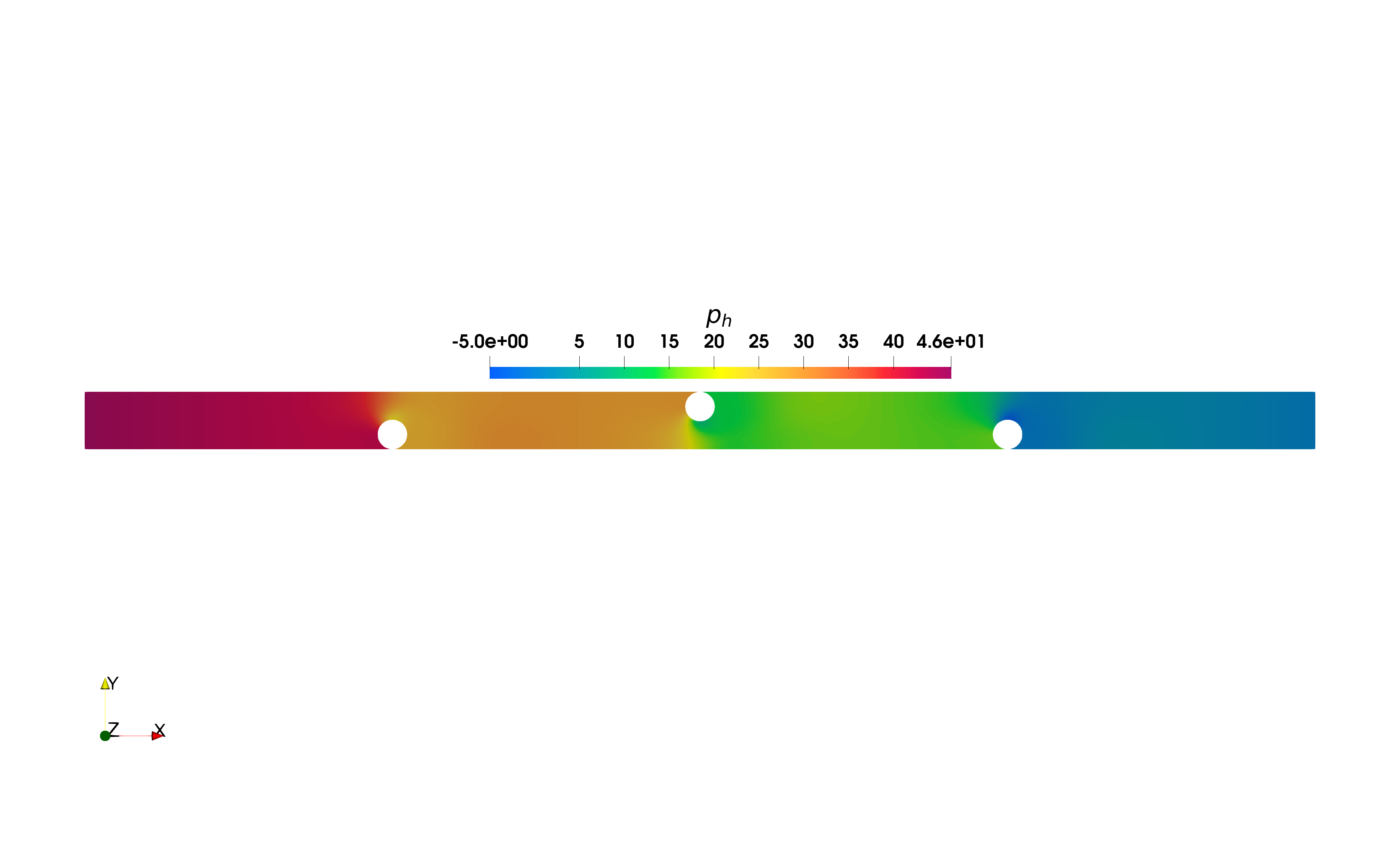}
		
	\end{minipage}
	\begin{minipage}{1\linewidth}
		\centering
		\footnotesize{$u_0=2.58\cdot10^{-1}m/s,\Delta P =5572875Pa $}
		\includegraphics[trim= 0cm 35cm 0cm 27cm,clip, scale=0.1]{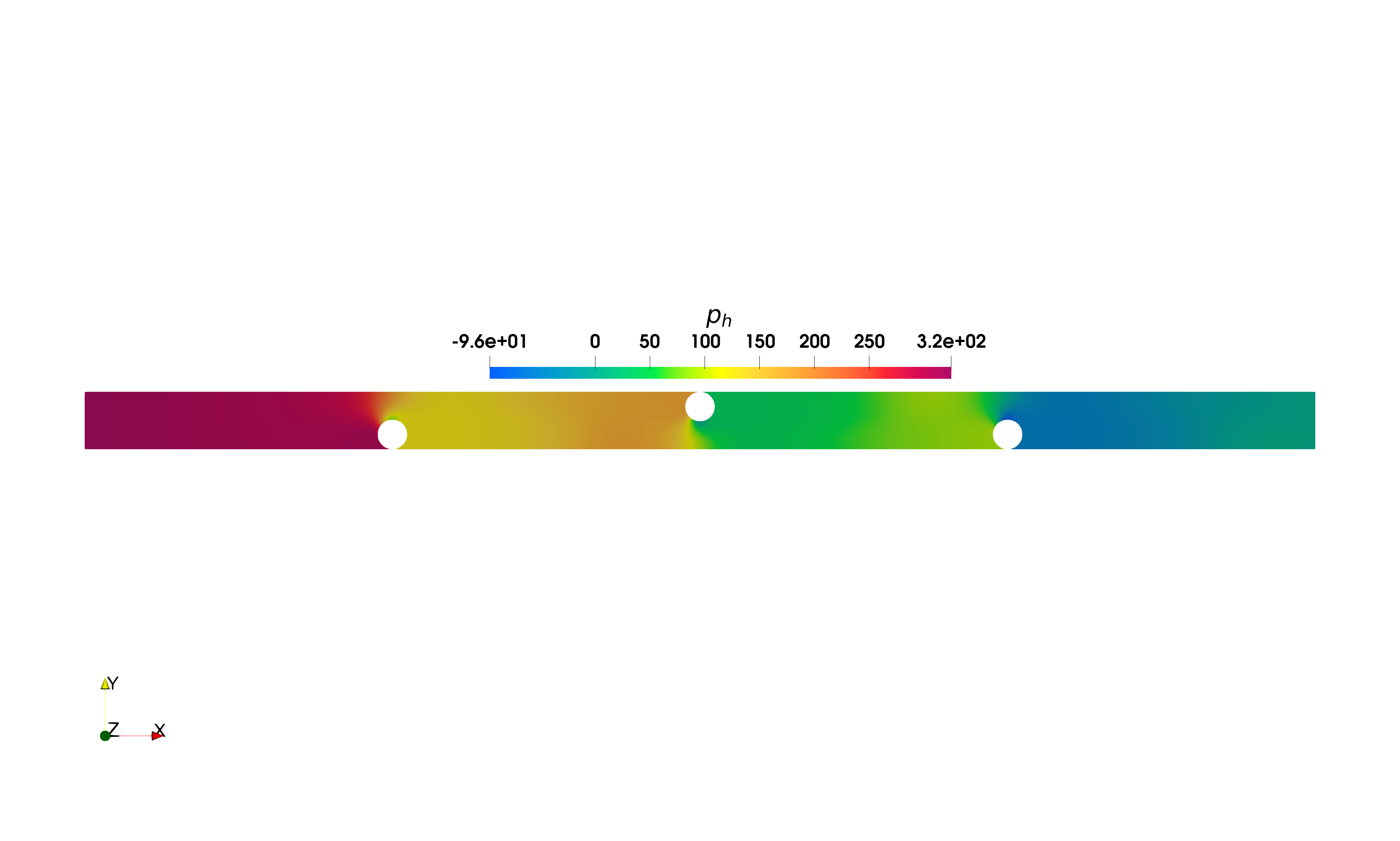}
		
	\end{minipage}
	\caption{Test 2. Comparison of relative pressure using a zig-zag spacers configuration.}
	\label{fig:zigzag3}
\end{figure}

\subsubsection{Submerged spacers configuration}
The numerical results for the spacers in submerged configuration are presented in Figures \ref{fig:submerge1}-\ref{fig:submerge3}. 

\begin{figure}[!ht]
	\centering
	\begin{minipage}{1\linewidth}
		\centering
		\footnotesize{$u_0=6.45\cdot10^{-2} m/s,\Delta P =4053000Pa $}
		\includegraphics[trim= 0cm 35cm 0cm 27cm,clip, scale=0.1]{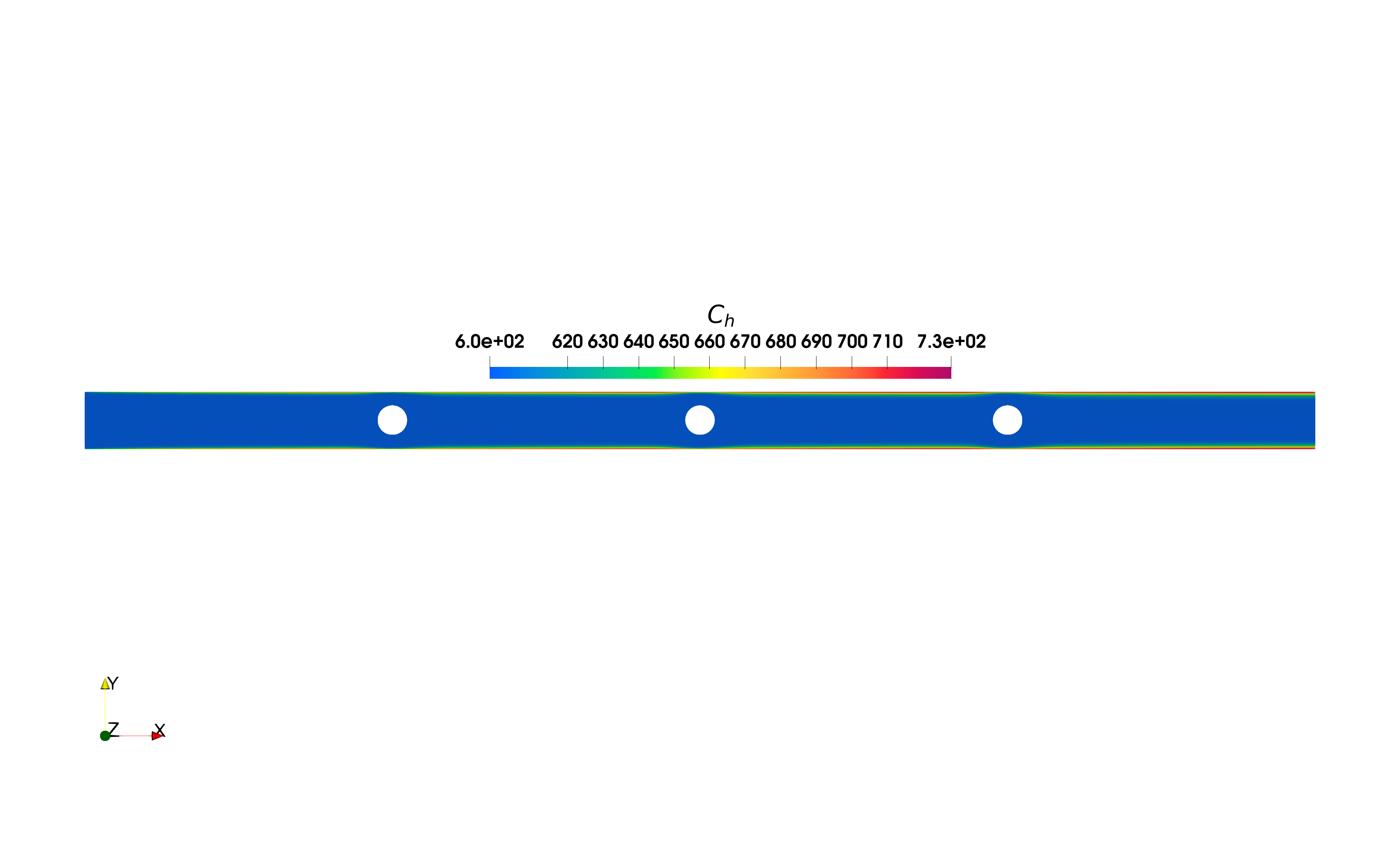}
		
	\end{minipage}
	\begin{minipage}{1\linewidth}
		\centering
		\footnotesize{$u_0=2.58\cdot10^{-1}m/s,\Delta P =4053000Pa $}
		\includegraphics[trim= 0cm 35cm 0cm 27cm,clip, scale=0.1]{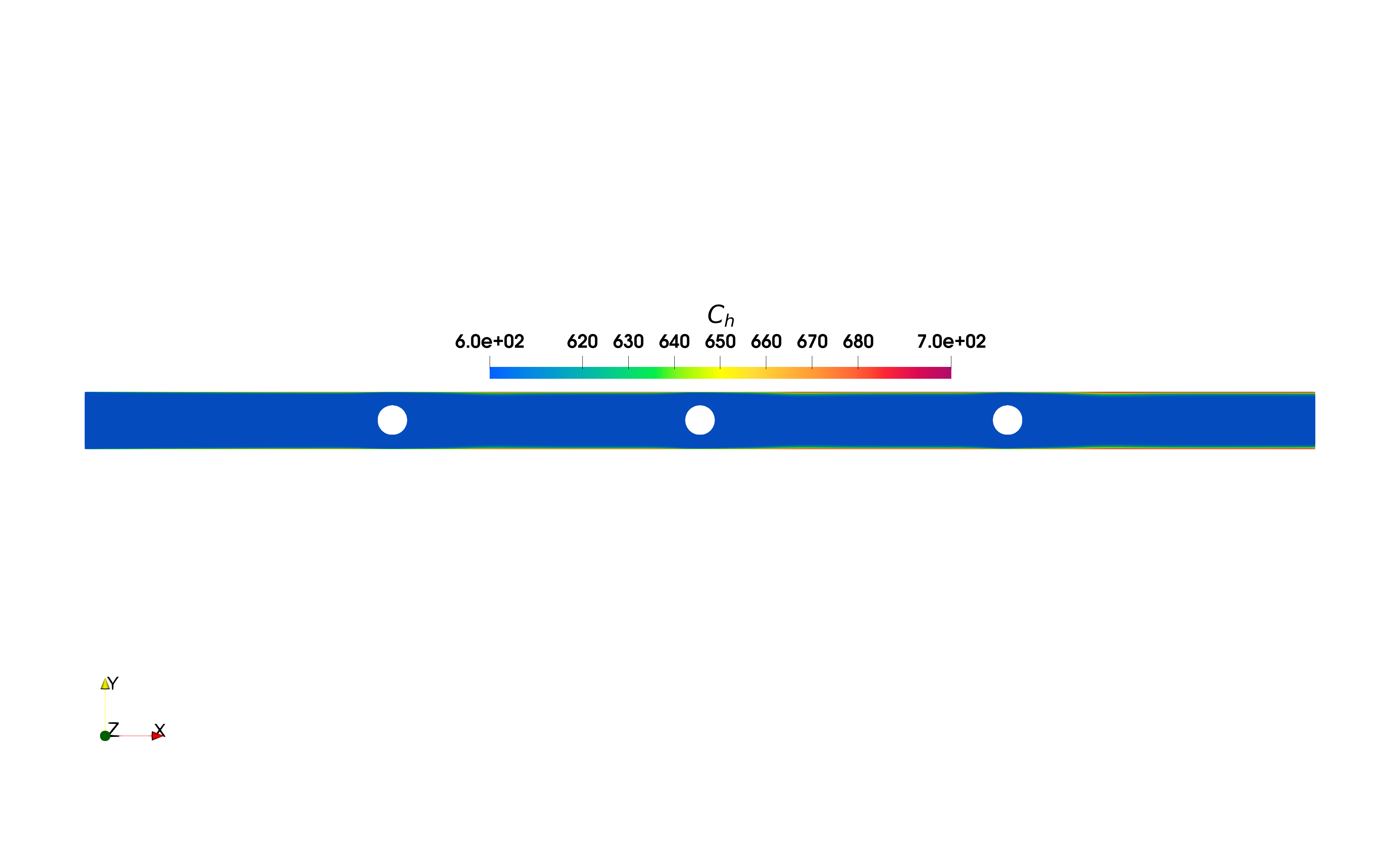}
		
	\end{minipage}
	\begin{minipage}{1\linewidth}
		\centering
		\footnotesize{$u_0=6.45\cdot10^{-2} m/s,\Delta P =5572875Pa $}
		\includegraphics[trim= 0cm 35cm 0cm 27cm,clip, scale=0.1]{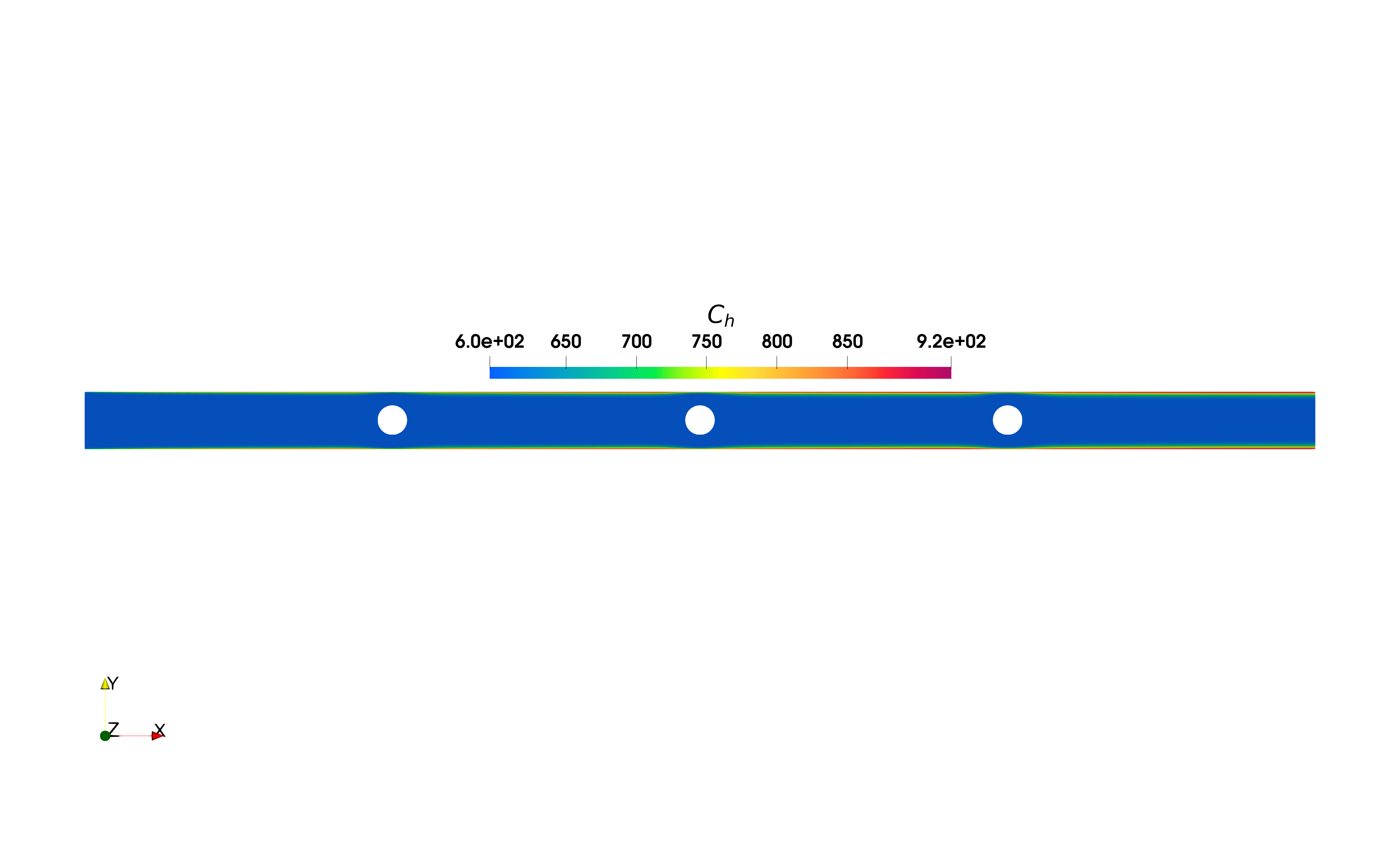}
		
	\end{minipage}
	\begin{minipage}{1\linewidth}
		\centering
		\footnotesize{$u_0=2.58\cdot10^{-1}m/s,\Delta P =5572875Pa $}
		\includegraphics[trim= 0cm 35cm 0cm 27cm,clip, scale=0.1]{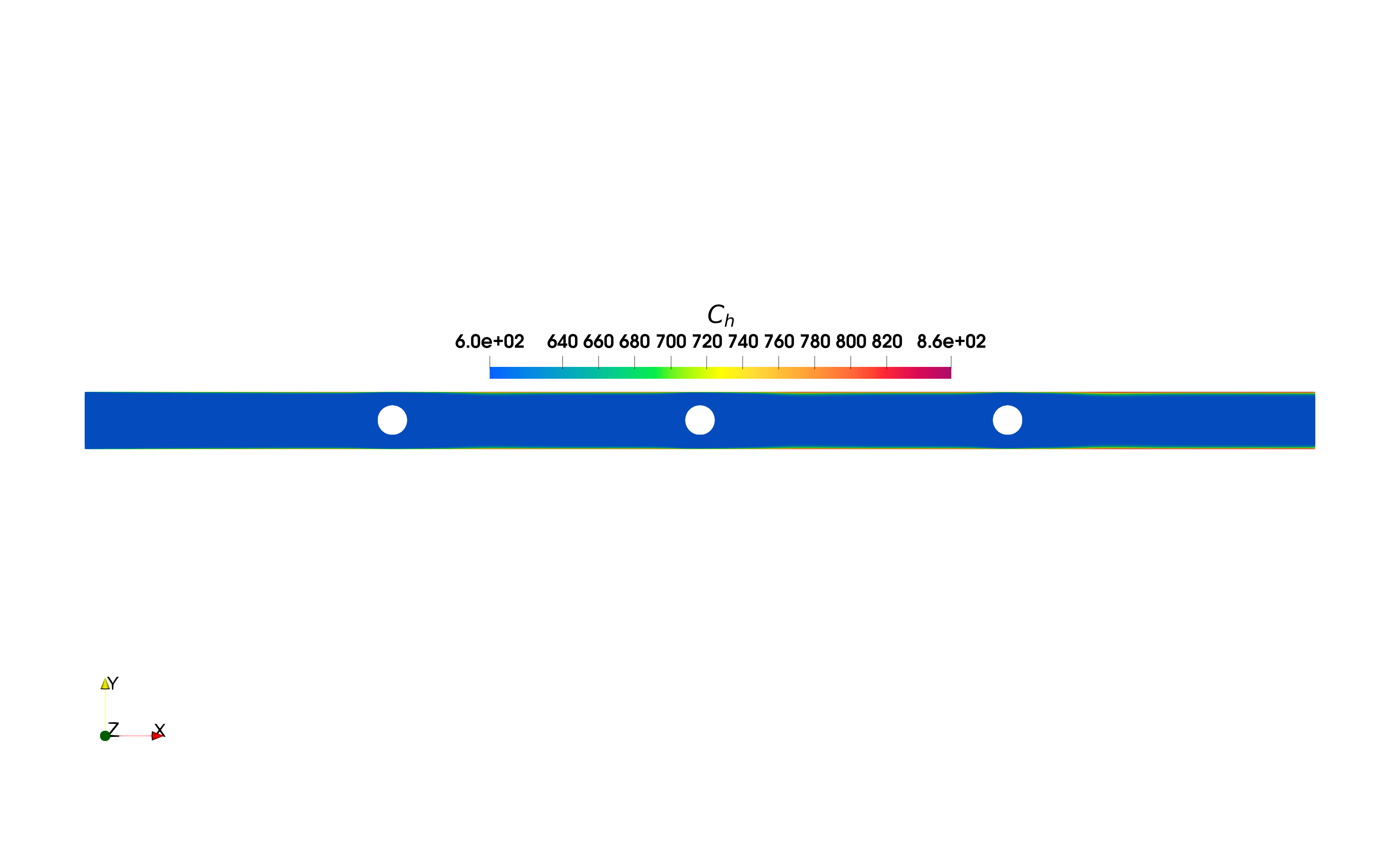}
		
	\end{minipage}
	\caption{Test 2. Comparison of concentration levels using a submerged spacers configuration.}
	\label{fig:submerge1}
\end{figure}

\begin{figure}[!ht]
	\centering
	\begin{minipage}{1\linewidth}
		\centering
		\footnotesize{$u_0=6.45\cdot10^{-2} m/s,\Delta P =4053000Pa $}
		\includegraphics[trim= 2.5cm 8cm 1.5cm 10cm,clip, scale=0.48]{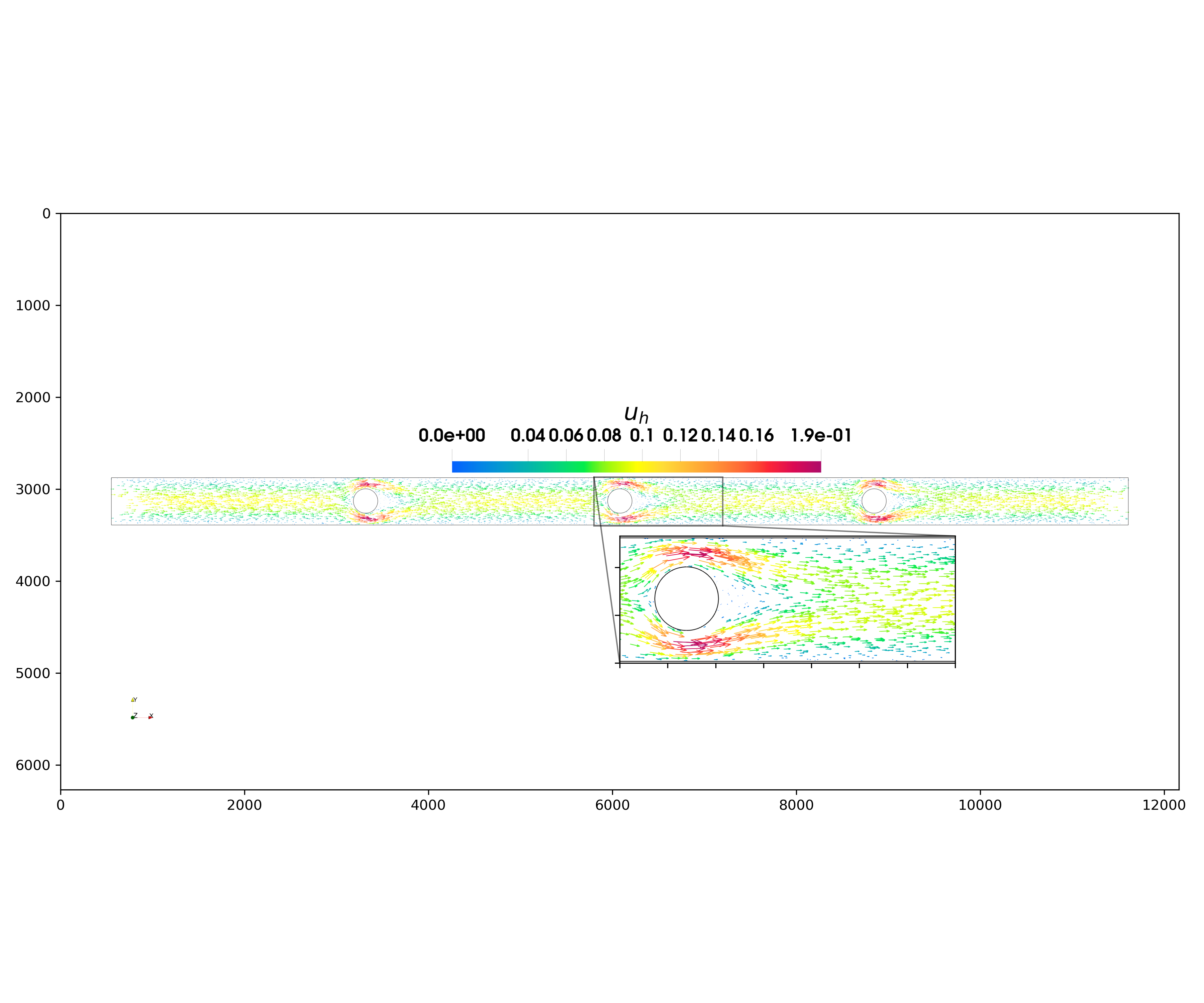}	
	\end{minipage}
	\begin{minipage}{1\linewidth}
		\centering
		\footnotesize{$u_0=2.58\cdot10^{-1}m/s,\Delta P =4053000Pa $}
		\includegraphics[trim= 2.5cm 8cm 1.5cm 10cm,clip, scale=0.48]{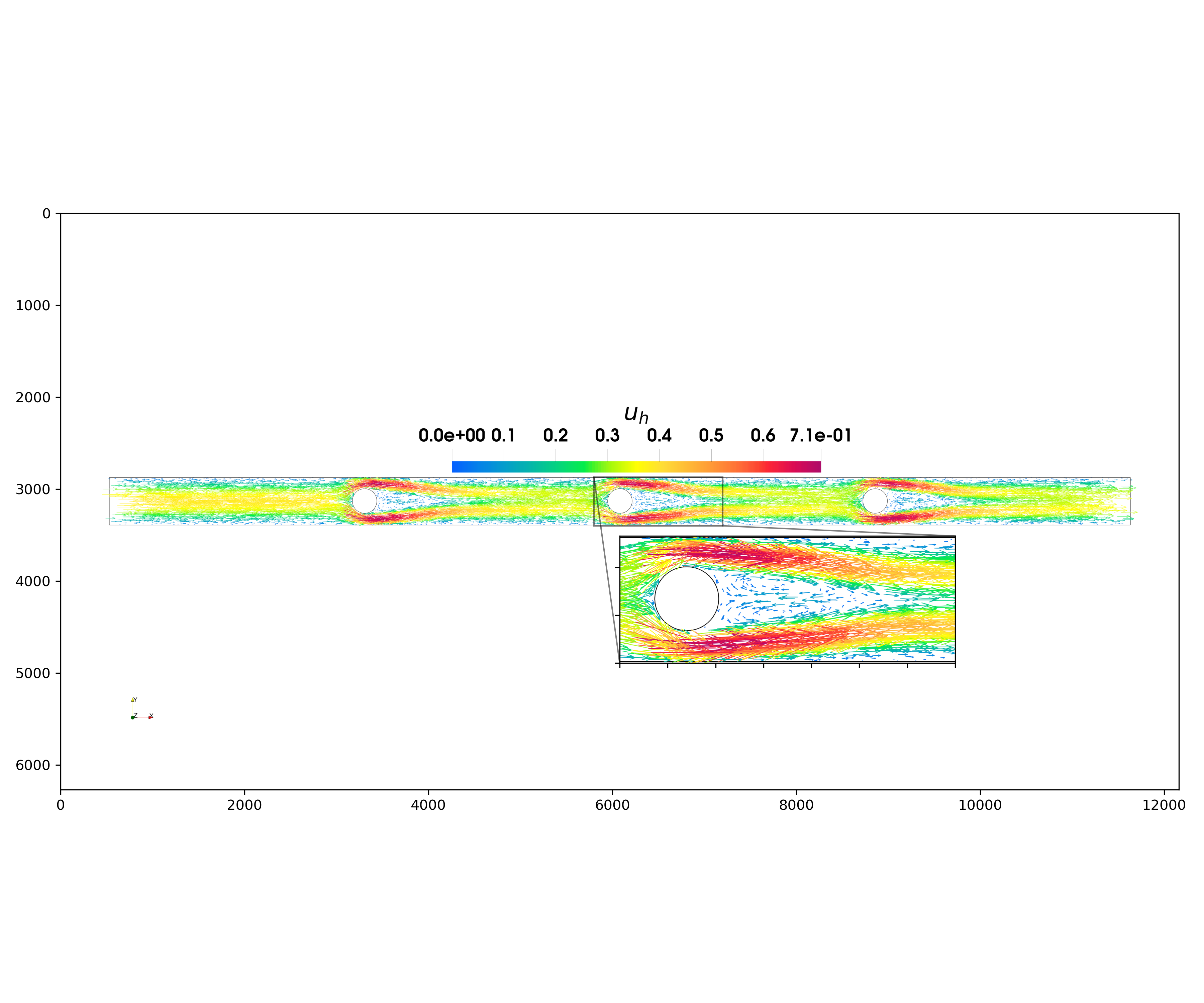}	
	\end{minipage}
	\caption{Test 2. Comparison of velocity fields using a submerged spacers configuration.}
	\label{fig:submerge2}
\end{figure}
\begin{figure}[!ht]
	\centering
	\begin{minipage}{1\linewidth}
		\centering
		\footnotesize{$u_0=6.45\cdot10^{-2} m/s,\Delta P =4053000Pa $}
		\includegraphics[trim= 0cm 35cm 0cm 27cm,clip, scale=0.1]{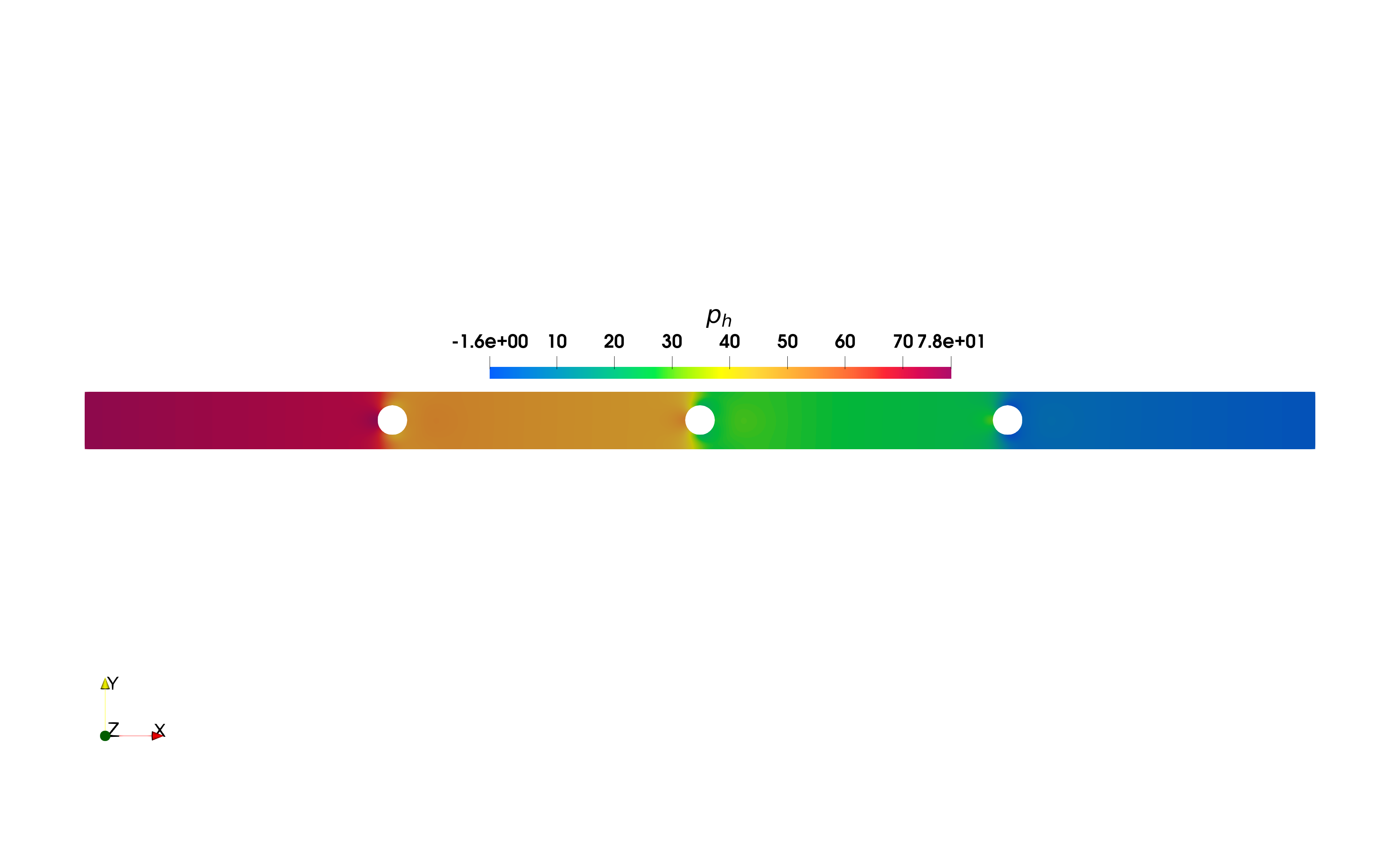}

	\end{minipage}
	\begin{minipage}{1\linewidth}
		\centering
		\footnotesize{$u_0=2.58\cdot10^{-1}m/s,\Delta P =4053000Pa $}
		\includegraphics[trim= 0cm 35cm 0cm 27cm,clip, scale=0.1]{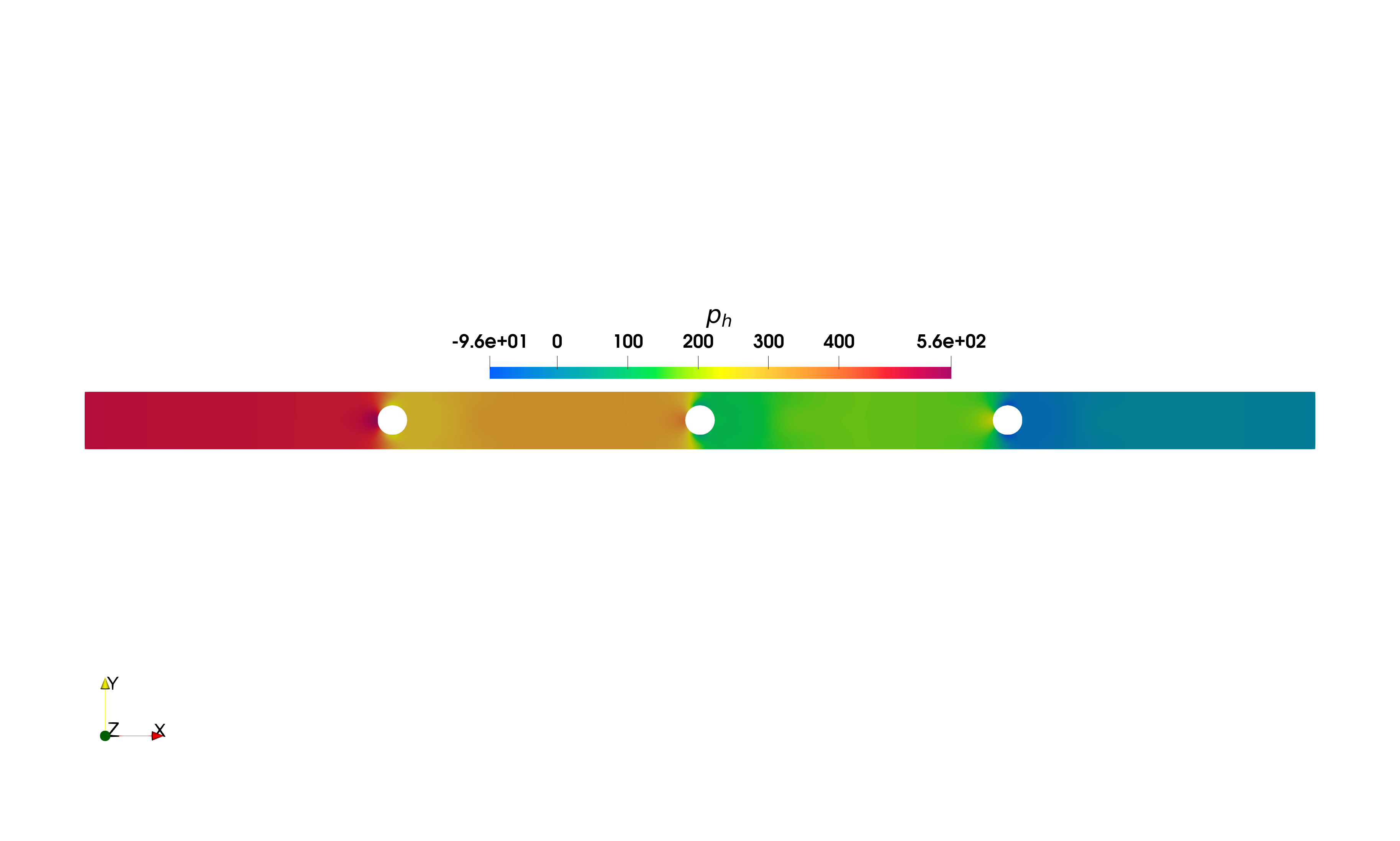}	
	\end{minipage}
	\begin{minipage}{1\linewidth}
		\centering
		\footnotesize{$u_0=6.45\cdot10^{-2} m/s,\Delta P =5572875Pa $}	
		\includegraphics[trim= 0cm 35cm 0cm 27cm,clip, scale=0.1]{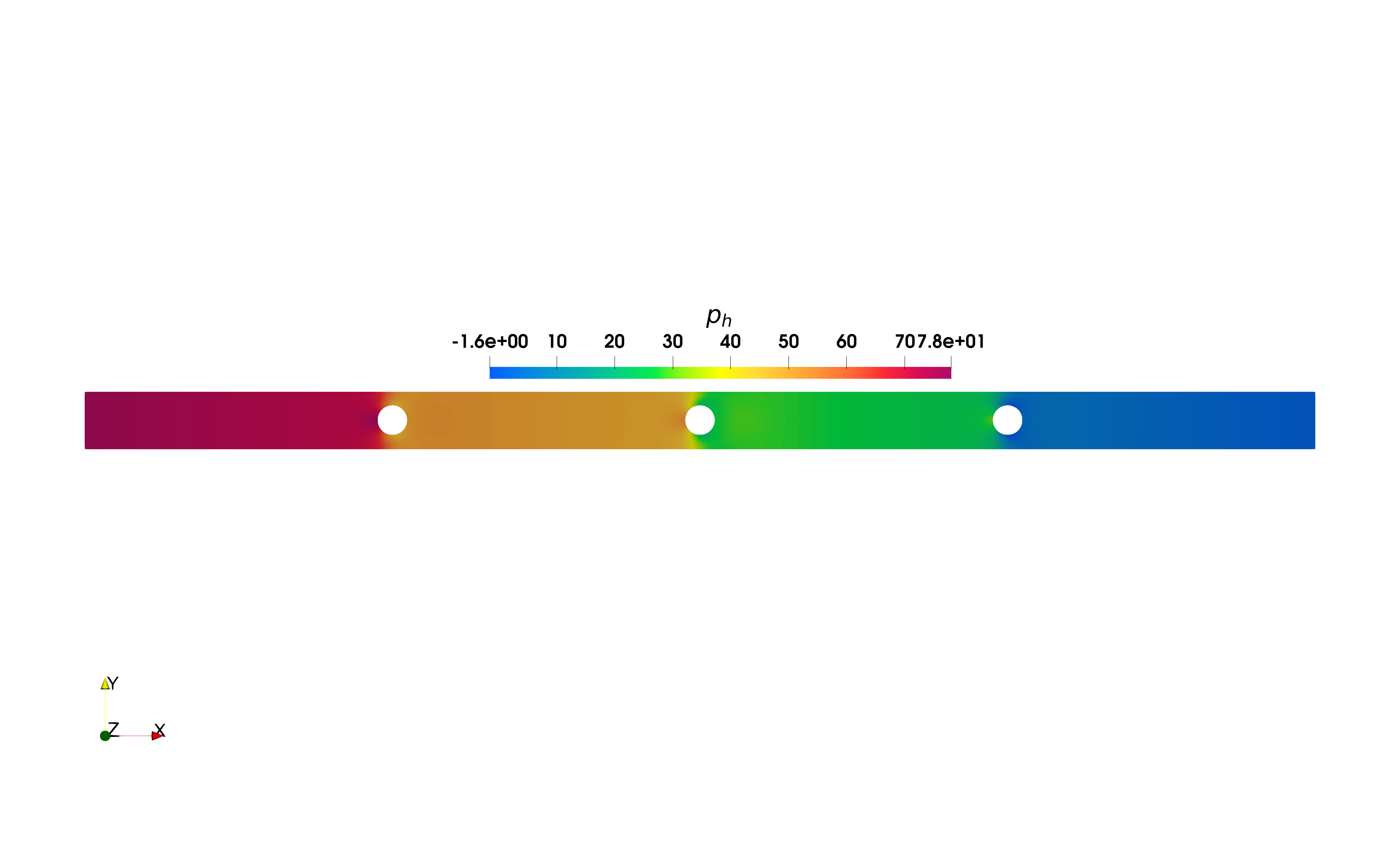}
	\end{minipage}
	\begin{minipage}{1\linewidth}
		\centering
		\footnotesize{$u_0=2.58\cdot10^{-1}m/s,\Delta P =5572875Pa $}	
		\includegraphics[trim= 0cm 35cm 0cm 27cm,clip, scale=0.1]{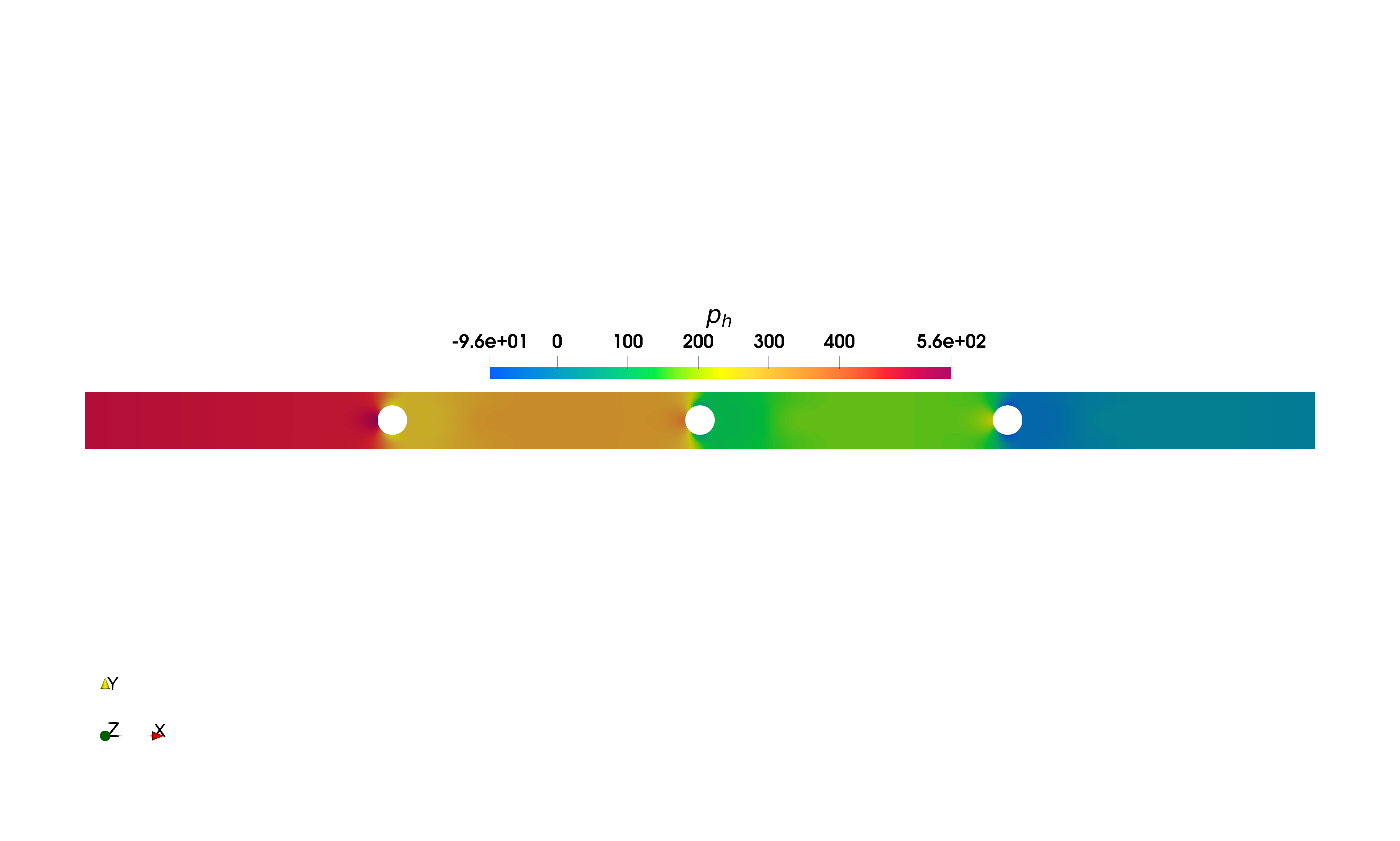}
	\end{minipage}
	\caption{Test 2. Comparison of relative pressure using a submerged spacers configuration.}
	\label{fig:submerge3}
\end{figure}
\newpage
\subsection{Discussion}
\label{sec:discusion}
This study focuses on the effect of three main parameters on the behavior of the membrane channel: the transmembrane pressure, the inlet velocity (i.e., the inlet mass flow) and the spacer configuration.

Analysis in Figures \ref{fig:cavity1}, \ref{fig:cavity2} and \ref{fig:cavity3} show that the increase in mass flow at the inlet (increase in inlet velocity) increases the length of the boundary layer and diminishes its concentration. Also we can establish that $\Delta P$ does not have any appreciable effect on the hydrodynamics of the channel for equal inlet mass flows. Indeed, the permeate flow difference between different $\Delta P$, although significative in terms of the obtained permeate flow, is two to three orders of magnitude smaller than the inlet velocity. These behaviors are further confirmed by Figure \ref{fig:wallcompcavity}.

In all of the performed cases, the main cause of pressure loss is the presence of a bottleneck caused by the spacers. As seen in Figure \ref{fig:cavity2}, the spacers cause a sudden augmentation of velocity and velocity gradients in those zones, which in turn derives into major pressure losses. The increment in inlet flow gets bigger as this gradient changes, and the pressure losses become greater. However, for increasing $\Delta P$ there is a subtle decrease in pressure loss, because the increment in permeate flux at higher $\Delta P$ decreases the amount of mass flow that has to cross the bottleneck, effectively reducing the pressure loss in each stage.

Both opposing behaviors leads to conclude that there is an optimum operating condition for the system from the energy efficiency point of view, in which both the increasing pressure loss due to increasing pressure flow and the reducing pressure loss from the increment in transmembrane pressure leads to an overall minimum pressure loss in the channel. This optimization has well established boundaries, as $\Delta P$ must be lower than the liquid entry pressure (LEP) of the membrane which would cause the membrane's rupture, but must be higher than the osmotic pressure of the feed solution, as otherwise the process would turn into Forward osmosis (FO). To determine these optimum values, further numerical essays need to be performed, which will be addressed in a subsequent article.

The previous analysis is the same for a zig-zag and submerged configurations upon comparing Figures \ref{fig:zigzag1}--\ref{fig:zigzag3} and \ref{fig:submerge1}--\ref{fig:submerge3}, respectively. However, for the case of a submerged configuration, the boundary layer of salt is not developed as seen in Figure \ref{fig:submerge1}, mainly due to the absence of recirculation zones close to the membrane walls, depicted in Figure \ref{fig:submerge2}. The recirculating fluid is located in the middle of the channel behind the spacers, away from the influence of the membrane wall concentration boundary layer (see Figure \ref{fig:submerge1}). Therefore, the effects of accumulation of salt in the recirculation zone will be greater along the channel, provided that the boundary layer is large enough.

In the case of pressure losses in submerged configurations, in contrast to the cavity and zig-zag configurations, has two bottlenecks in each spacer that are smaller, causing more pressure loss inside the channel. As the bottlenecks are reduced to half the size, both the velocity gradients and the convective flux at the bottenecks increase, therefore increasing the pressure losses by viscous friction and convective losses.

The main variables studied in this paper for all conditions are shown in Figures \ref{fig:wallcompcavity} through \ref{fig:wallcompconfigs}. From inspection of the permeate velocity curves in Figures \ref{fig:wallcompcavity}--\ref{fig:wallcompsubmerge}, we can say that incrementing the inlet mass flow helps to diminish the concentration in the whole membrane, therefore reducing the osmotic pressure effect in the permeate flux. Also, the increase in pressure leads to an increase in concentration at the membrane, as more water is being separated out, leaving behind more salt. Also, as expected, the permeate flux increases when $\Delta P$ is increased.

Pressure profiles in Figures \ref{fig:wallcompcavity}--\ref{fig:wallcompsubmerge} show the pressure loss along the membrane channel for the cavity configuration for all different systems. The two main conclusions are that $\Delta P$ has a negligible effect on the pressure loss of the system, as stated earlier in the discussion, and the increment in the inlet velocity has a positive contribution in the pressure loss. Indeed, for twice the inlet velocity, the pressure losses are increased by a factor greater than two. The negative pressure values in all Figures are valid due to the incompressible formulation of the problem, which allows to write the variable $p$ in terms of the gauge pressure.

\begin{figure}[!h]
	\centering
	\includegraphics[scale=0.25]{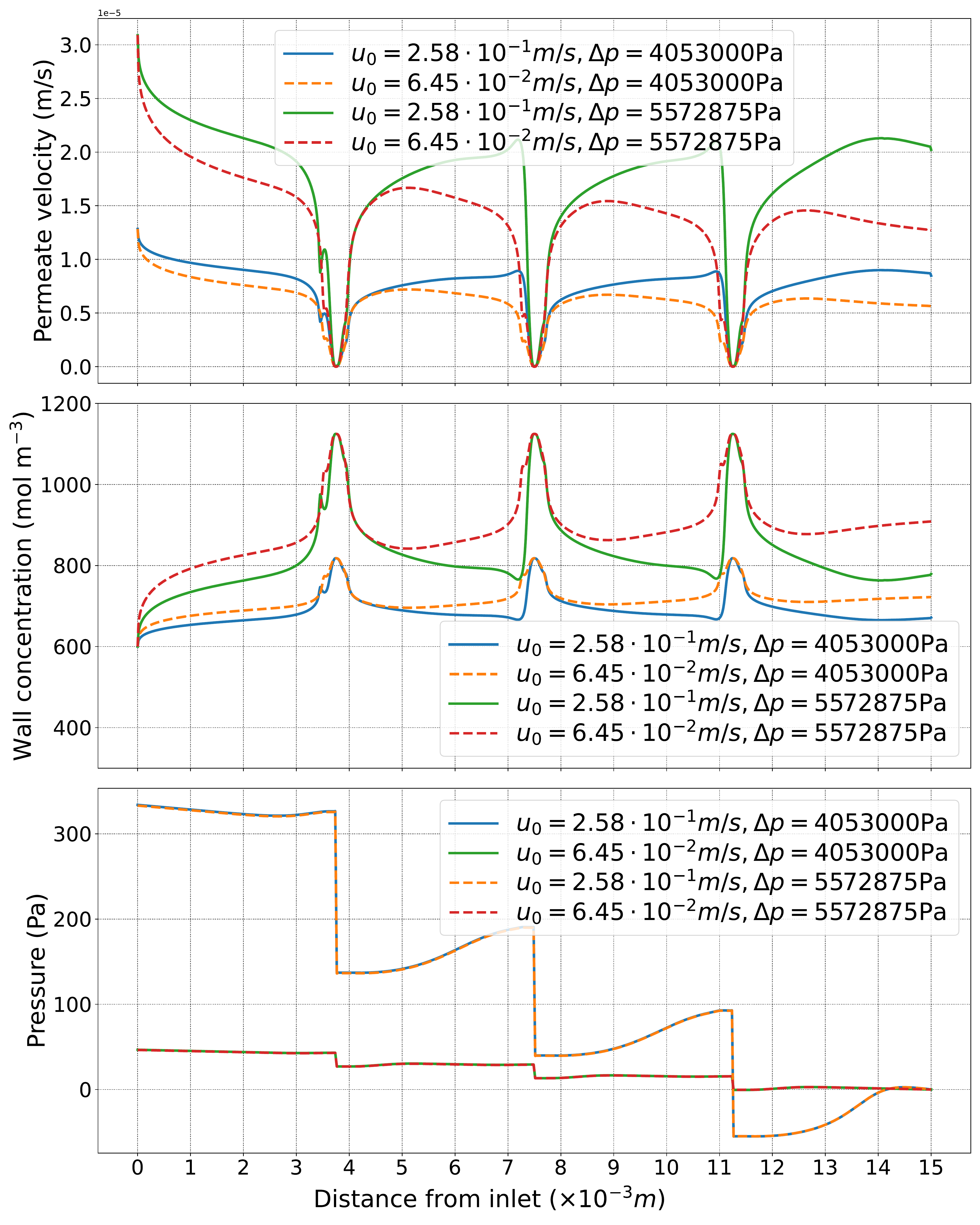}
	\caption{Test 1. Comparison between permeate velocity, concentration level and relative pressure drop on the membrane walls along the line $y=0$ and $x\in[0,L]$ using the cavity spacers configuration.}
	\label{fig:wallcompcavity}
\end{figure}
\begin{figure}[!h]
	\centering
	\includegraphics[scale=0.25]{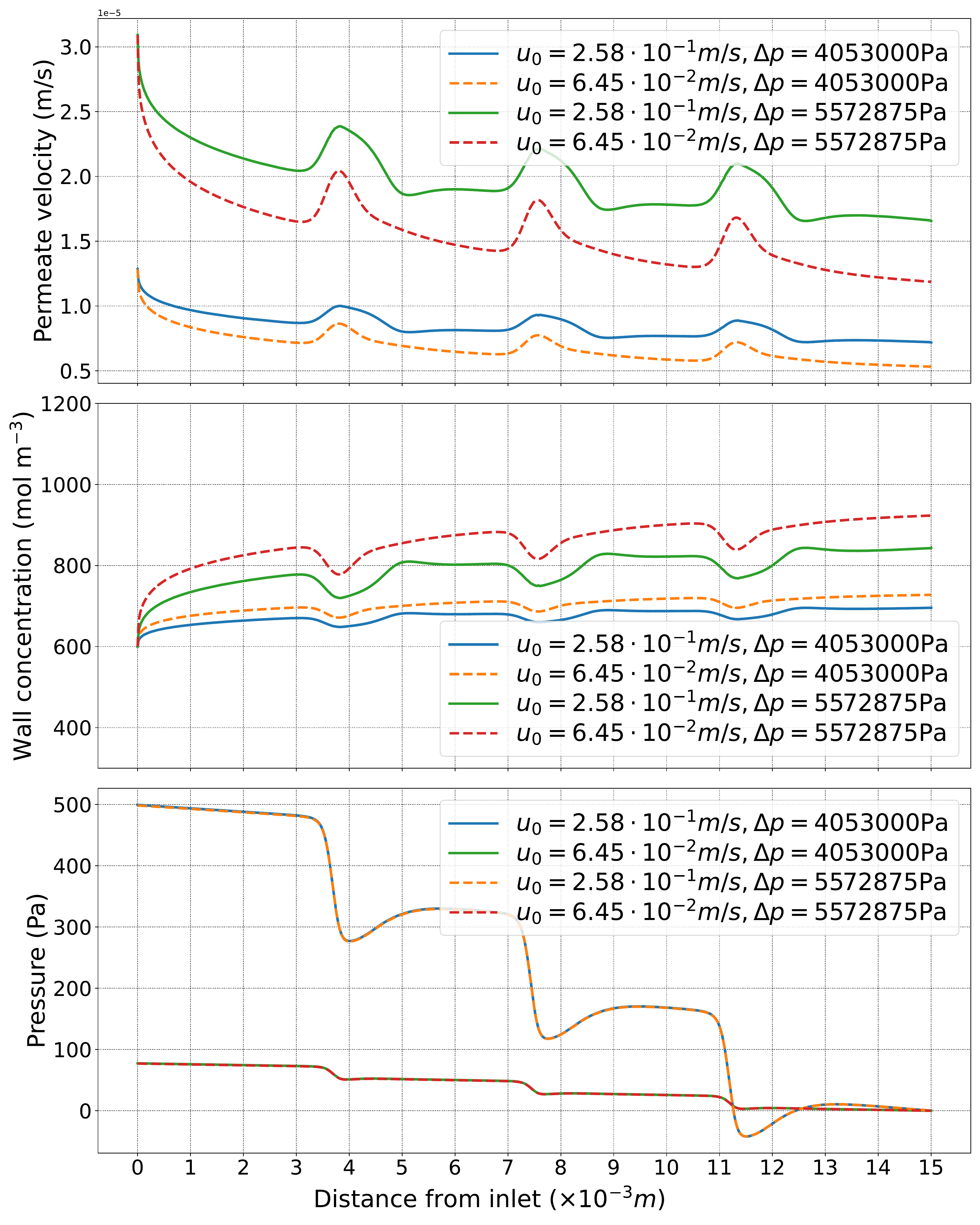}
	\caption{Test 1. Comparison between permeate velocity, concentration level and relative pressure drop on the membrane walls along the line $y=0$ and $x\in[0,L]$ using the submerged spacers configuration.}
	\label{fig:wallcompsubmerge}
\end{figure}
\begin{figure}[!h]
	\centering
	\includegraphics[scale=0.25]{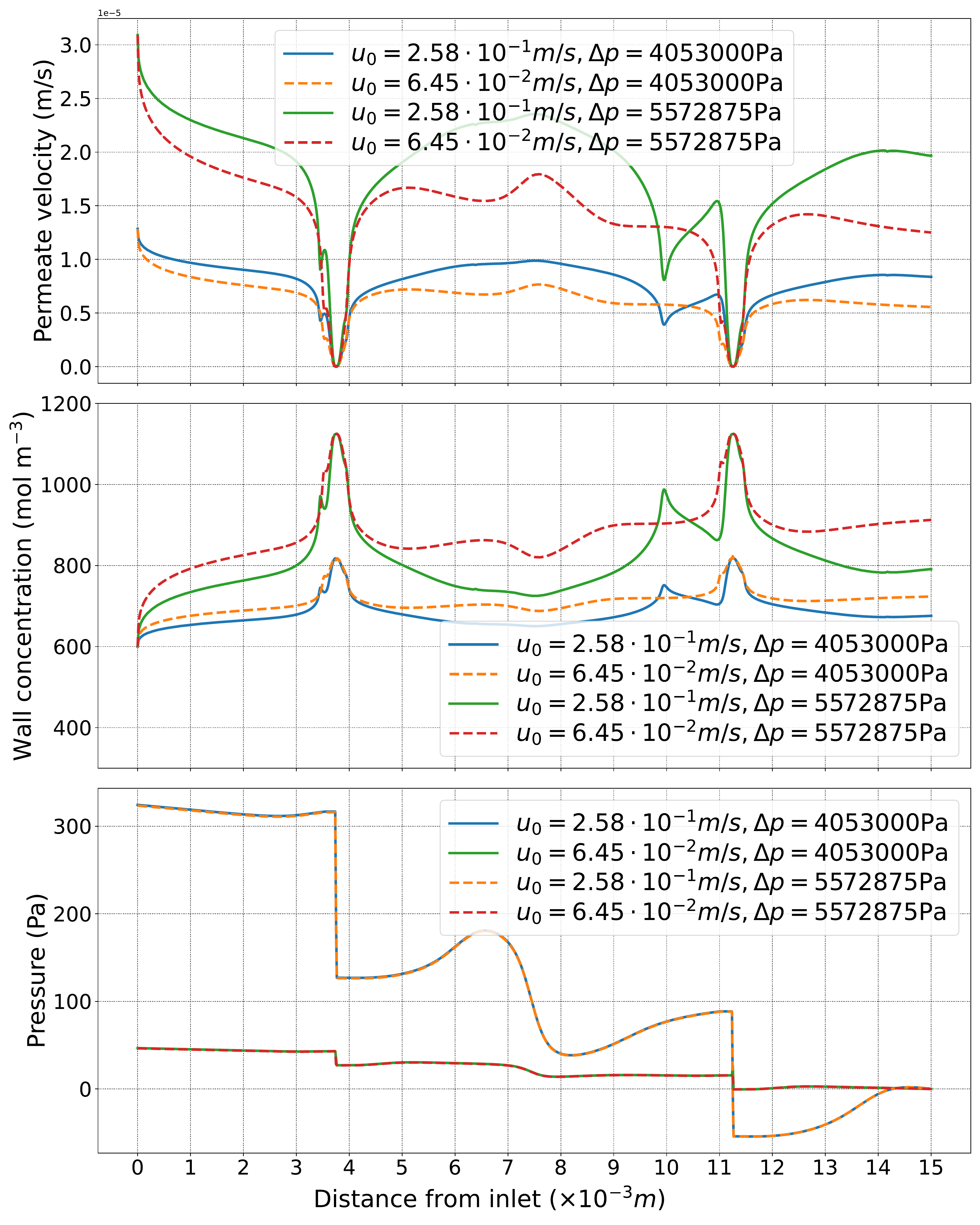}
	\caption{Test 1. Comparison between permeate velocity, concentration level and relative pressure drop on the membrane walls along the line $y=0$ and $x\in[0,L]$ using the zig-zag spacers configuration.}
	\label{fig:wallcompzigzag}
\end{figure}

%

\begin{figure}[!h]
	\centering
	\includegraphics[scale=0.25]{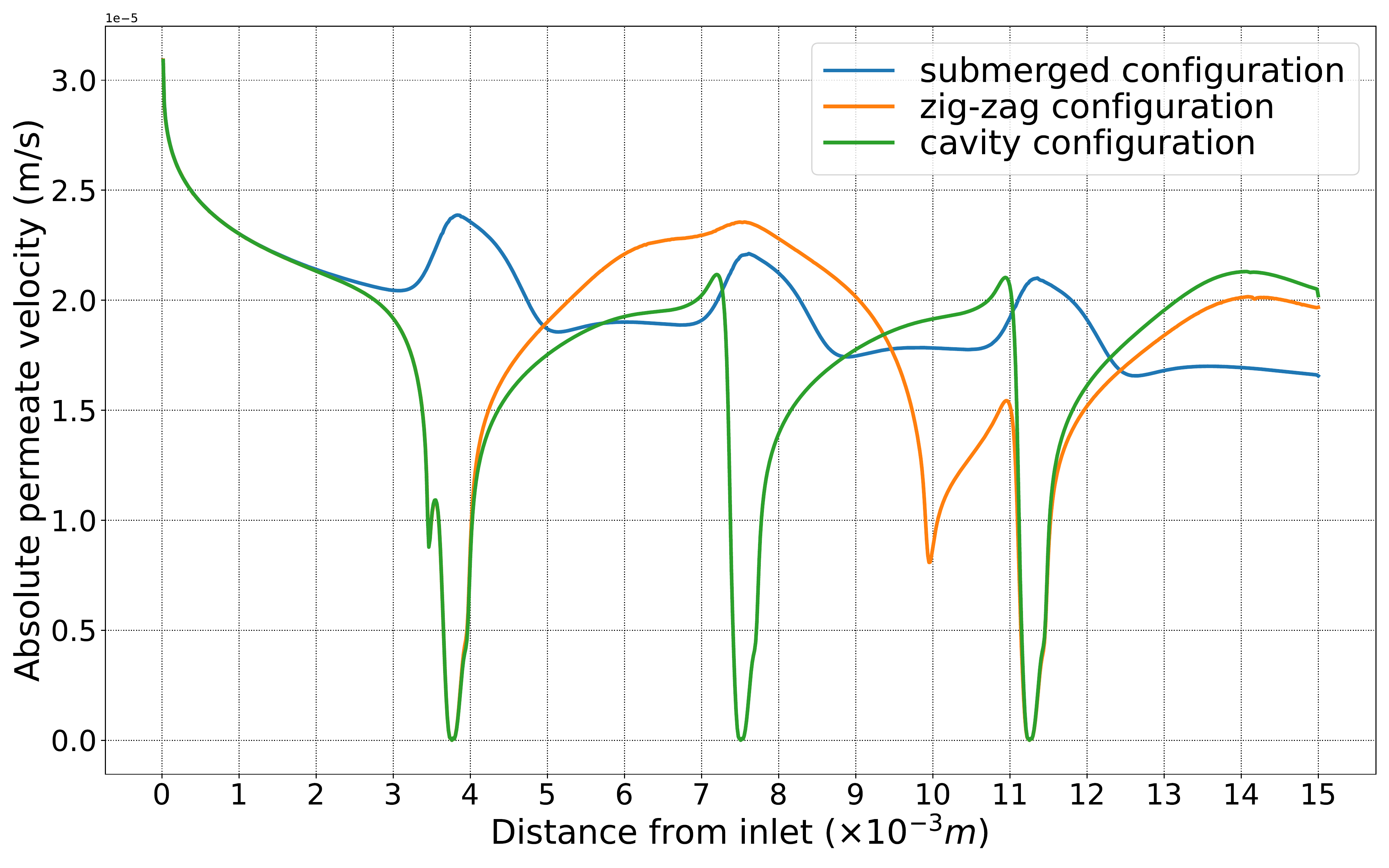}
	\caption{Comparison between absolute permeate of different configurations along the line $y=0$ and $x\in[0,L]$.}
	\label{fig:wallcompconfigs}
\end{figure}
Additionally, the effect of spacer configuration on the membrane's performance needs to be analyzed. For this, we fix the inlet velocity and transmembrane pressure at $2.58\cdot 10^{-1}\,m/s$ and $\Delta P = 5572875\,Pa$, respectively. From Figure \ref{fig:wallcompconfigs}, it can be concluded that spacers in both cavity and zig-zag configurations have a negative impact on permeate velocity on their vicinity, as shown by the sharp decrease in the permeate velocity at the positions where the spacers are located. The opposite effect occurs for the submerged configuration, where a local maximum can be located at each spacer position. This is due to the increase in flow in the channel bottlenecks, which decrease the salt concentration and consequently the osmotic resistance to permeate flux. However, the mixing effects are less effective than the other two configurations, as the concentration at the end of the submerged configuration channel is higher and the permeate flux is the lowest of all three configurations. We can also infer that the high inlet velocity in the zig-zag configuration as seen in Figure \ref{fig:zigzag2} produces a dominant fluid path, which diminishes the concentration between the first two spacers and therefore increases the overall permeate flux.
Lastly, the total volumetric flow per unit width for all membrane channels in different configurations, for the same conditions used in Figure \ref{fig:wallcompconfigs}, is given in Table \ref{tab:fluxcomp}. This quantity can be obtained from direct integration of the definition of volumetric flow from the flux:
\begin{equation}
		\label{eq:volumetric-flow}
		\frac{\dot{V}}{W}=\int_{\Gamma_{m}}\abs{u_{y}(s)} ds,
\end{equation}
where $\dot{V}$ is the total volumetric flow of the permeate from both membrane walls, $W$ is the width of the channel, and $u_{y}\left(s\right)$ is the permeate flux at a distance from the inlet $s$. The absolute value is used because on the lower side of the membrane the velocity vector is negative with respect to the coordinate origin.

\begin{table}[!h]
\caption{Total permeate flow per unit length $\dot{V}/W$ for all membrane channel simulations.}
	\label{tab:fluxcomp}
\centering
\begin{tabular}{llll}
	\hline
	Configuration & $u_{0}\,[m/s]$ & $\Delta P\,[Pa]$ & $\dot{V}/W\,[m^{3}/(s\cdot m)]$ \\ \hline
	Cavity        & 0.258          & 4053000          & $2.39347\cdot 10^{-7}$          \\
	Cavity        & 0.129          & 4053000          & $2.16131\cdot 10^{-7}$          \\
	Cavity        & 0.064          & 4053000          & $1.93002\cdot 10^{-7}$          \\
	Cavity        & 0.258          & 5572875          & $5.59461\cdot 10^{-7}$          \\
	Cavity        & 0.129          & 5572875          & $5.00098\cdot 10^{-7}$          \\
	Cavity        & 0.064          & 5572875          & $4.41913\cdot 10^{-7}$          \\
	Submerged     & 0.258          & 4053000          & $2.51089\cdot 10^{-7}$          \\
	Submerged     & 0.129          & 4053000          & $2.28837\cdot 10^{-7}$          \\
	Submerged     & 0.064          & 4053000          & $2.03092\cdot 10^{-7}$          \\
	Submerged     & 0.258          & 5572875          & $5.88266\cdot 10^{-7}$          \\
	Submerged     & 0.129          & 5572875          & $5.31309\cdot 10^{-7}$          \\
	Submerged     & 0.064          & 5572875          & $4.66112\cdot 10^{-7}$          \\
	Zig-zag       & 0.258          & 4053000          & $2.38891\cdot 10^{-7}$          \\
	Zig-zag       & 0.129          & 4053000          & $2.16246\cdot 10^{-7}$          \\
	Zig-zag       & 0.064          & 4053000          & $1.93055\cdot 10^{-7}$          \\
	Zig-zag       & 0.258          & 5572875          & $5.57996\cdot 10^{-7}$          \\
	Zig-zag       & 0.129          & 5572875          & $5.01593\cdot 10^{-7}$          \\
	Zig-zag       & 0.064          & 5572875          & $4.42017\cdot 10^{-7}$          \\ \hline
\end{tabular}
\end{table}
From the values of Table \ref{tab:fluxcomp} it follows that the submerged configuration has the most production of permeate of all three configurations for all of the considered variations of $u_{0}$ and $\Delta P$. Also, zig-zag and cavity configurations show similar results, with zig-zag being slightly more efficient than its alternate counterpart. However, for longer channels this tendency may be inverted, as the submerged configuration concentration profile may rise well above the other configurations, therefore diminishing its permeate flux. The results also show that quadrupling the value of $u_{0}$ while mantaining $\Delta P$ can generate an increase in $\dot{V}/W$ of up to $24\%$ for the lower pressure and $26\%$ for the higher pressure. Additionally, increasing the pressure from $4053000\,Pa$ to $5572875\,Pa$ drastically rises the value up to about $130\%$ for the lower inlet velocity and up to $134\%$ for the higher value.
\section{Conclusions}
In this work we have presented the performance of the numerical method for a RO model using the Nitsche technique. The method has been successfully validated by testing the accuracy with respect to the pressure drop, as well as the number of cells required to have independence with respect to the mesh size. However, for the highest speed of operation, it was required to implement the SUPG scheme to stabilize the algorithm.

Also, we can follow the same approach in this study to obtain the classical penalty formulation, which also serves to impose the permeability condition. However, it is well known that, for very large values of $\alpha$, the system will be ill conditioned, and consequently, the convergence of the method will be affected \cite{barrett1986finite}.

The obtained tendencies for permeate flux against transmembrane pressure and inlet mass flow are in accordance with previous works. It was also determined that the bottlenecks caused by the spacers were the main contributors to pressure loss (and therefore, energy loss), being the biggest drop associated with the submerged configuration. Considering all process parameters constant, the greatest permeate flux was obtained for the submerged configuration, although the higher concentration at the exit in comparison with its counterparts suggest that this advantage may not hold for bigger channels, as the submerged configuration is less effective at disrupting the boundary layer at the membrane walls.\\ 
For future development of the proposed model, further studies are needed to establish the influence of the Nitsche technique on the stability of the method. Additionally, other factors must be considered to accurately capture the complex behavior of a membrane module in operational conditions, e.g., the inclusion of fouling agents and the effects of three-dimensional spacer mixing.
\bibliographystyle{siam}
\bibliography{biblio}
\end{document}